\PassOptionsToPackage{reqno}{amsmath}
\documentclass[12pt]{amsart}

\usepackage{amsmath,amscd,amssymb,amsfonts,amsthm,hyperref}
\usepackage{mdframed}
\usepackage{xcolor}  
\hypersetup{   colorlinks,   linkcolor={red!50!black},    citecolor={blue!70!black},   urlcolor={blue!80!black}}

\usepackage{tcolorbox}
\usepackage{amsmath,amssymb}
\usepackage{mathrsfs}
\usepackage{comment}

\usepackage[
  a4paper,
  top=2.4cm,
  bottom=2.8cm,
  left=2.6cm,
  right=2.6cm
]{geometry}

\usepackage{amsmath,amssymb,amsthm,eucal,manfnt}
\usepackage[normalem]{ulem}
\usepackage{color, tikz}
\usepackage{tikz-cd,tikz-3dplot}
\usetikzlibrary{calc,babel}
\usetikzlibrary{arrows,decorations.markings}
\usepackage[all]{xy}
\usepackage{tensor}
\usepackage{todonotes}
\usepackage{typed-checklist}
\usepackage{comment}
\title[]{On the equivalence of the integrability obstructions for transitive Lie algebroids
}

\author{Paolo Antonini, Alessio Giannotta}

\email{paolo.antonini@unisalento.it, alessio.giannotta@unisalento.it}
\address{}
\keywords{}
\thanks{}
\subjclass[]{}

\numberwithin{equation}{section} 

\theoremstyle{plain} 
\newtheorem{thm}{Theorem}[section]
\newtheorem{lemma}[thm]{Lemma}
\newtheorem{prop}[thm]{Proposition}
\newtheorem{corl}[thm]{Corollary}

\theoremstyle{definition} 
\newtheorem{defn}[thm]{Definition}

\theoremstyle{remark} 
\newtheorem{rmk}[thm]{Remark}
\newtheorem{exm}[thm]{Example}

\newcommand{\g}{\mathfrak{g}}

\newcommand{\rest}[1]{\big\rvert_{#1}} 
\newcommand{\R}{\mathbb{R}}   

\def\pairL_#1(#2|#3){\mbox{${}_{#1}({#2}\mid{#3})$}} 
\def\pairR(#1|#2)_#3{\mbox{$({#1}\mid{#2})_{#3}$}} 
\def\scal<#1|#2>{\langle#1\mid#2\rangle} 

\newbox\ncintdbox \newbox\ncinttbox 
	\setbox0=\hbox{$-$}
	\setbox2=\hbox{$\displaystyle\int$}
	\setbox\ncintdbox=\hbox{\rlap{\hbox
		to \wd2{\hskip-.125em \box2\relax\hfil}}\box0\kern.1em}
	\setbox0=\hbox{$\vcenter{\hrule width 4pt}$}
	\setbox2=\hbox{$\textstyle\int$}
	\setbox\ncinttbox=\hbox{\rlap{\hbox
		to \wd2{\hskip-.175em \box2\relax\hfil}}\box0\kern.1em}

\begin{document}

\begin{abstract}
The integrability problem for transitive Lie algebroids can be looked at from different perspectives, revealing an interplay between cohomological methods and homotopical constructions. Mackenzie \cite{zbMATH00822678} introduced a cohomological obstruction defined via sheaf-theoretic methods. On the other hand, Crainic and Fernandes \cite{zbMATH02057953} used a path space approach and characterized integrability in terms of the monodromy. Recently, Meinrenken \cite{zbMATH07480966} formulated the monodromy in terms of a clutching construction.

We show that all of these agree. In particular, we identify the monodromy map with the Mackenzie obstruction class through the natural pairing between cohomology and homotopy.
\end{abstract}

\maketitle

%
%



\section{Introduction}
  Let $A \Rightarrow M$ be a transitive algebroid over a (simply connected) manifold $M$.  The integrability problem for $A$ can be approached from two complementary perspectives: a cohomological one, given by the obstruction class $e^A \in \check{H}^2(M;Z\widetilde{G})$ introduced by Mackenzie \cite{zbMATH00822678}, and an homotopical one, based on the monodromy map $\partial^A:\pi_2(M) \to Z\widetilde{G}$ constructed by Crainic and Fernandes  \cite{zbMATH02057953}.

In both cases, $\widetilde{G}$ is the simply connected Lie group integrating the adjoint Lie algebra at a fixed reference point $m\in M$ and $Z\widetilde{G}$ is its center. \smallskip 

In this note we identify these two obstructions via the natural pairing between cohomology and homotopy. In doing so, we relate each of them to the description of the monodromy map given by Meinrenken \cite{zbMATH07480966} via a clutching construction. \smallskip

We briefly recall the three approaches referring to the already existing literature for the details.

Mackenzie \cite{zbMATH00822678} constructed a class $e^A \in \check{H}^2(M;Z\widetilde{G})$ in the \v{C}ech cohomology of the base with coefficients in $Z\widetilde{G}$. This can be viewed as a non abelian generalization of the Chern class of a $U(1)$-bundle and completely controls integrability: the algebroid $A$ is integrable if and only if $e^A$ lies in the image of a discrete subgroup of $Z\widetilde{G}$. On the other hand Crainic and Fernandes \cite{zbMATH02057953} solved the integrability problem for any Lie algebroid. In the transitive, integrability is encoded by the image of a 
{\em{monodromy map}}, a
group morphism $\partial^A:\pi_2(M) \to Z\widetilde{G}$. The algebroid is integrable if and only if the image of $\partial^A$ is discrete. \smallskip

Mackenzie's methods are of local to global type. First he proves that transitive algebroids are locally trivial, then constructs a class preventing the globalization of the local integrations.
Crainic and Fernandes methods are purely global, they use a path space approach inspired by the Poisson sigma model \cite{zbMATH01686785,zbMATH01628538} and Sullivan's rational homotopy theory. The monodromy map is related to the failure of the smoothness of the Weinstein groupoid of $A$, a topological groupoid, which in a sense, already integrates $A$.
Finally Meinrenken has recently given a description of the monodromy map generalizing the clutching construction for principal bundles with connection \cite{zbMATH07480966}. \smallskip

Before stating the result identifying all these obstructions we look at the example of the transitive algebroid appearing in the geometric prequantization problem \cite{zbMATH05024689,zbMATH07456034,zbMATH06931092}.

Let $(M, \omega)$ be a symplectic manifold. Consider the Lie algebroid $A_\omega$ given by the vector bundle $A_\omega=TM\times\mathbb{R}=TM\oplus (M\times \mathbb{R})$, the projection onto $TM$ as anchor map, and Lie bracket on sections:
$$[X\oplus\xi,Y\oplus\eta]=[X,Y]\oplus X(\eta)-Y(\xi)+\omega(X,Y).$$
Fix a good cover $\{U_i\}_{i\in I}$ of $M$. The Mackenzie's construction
first yields a family of Lie algebroid trivializations $\Psi_i\colon TU_i\times\R\to \left. A_\omega\right|_{U_i}=TU_i\times\R$ in the form
$$\Psi_i(X\oplus \xi)=X\oplus\theta_i(X)+\xi.$$
Since these preserve the bracket, each 1-form $\theta_i$ (locally) integrates the curvature of $A_\omega$, i.e.
$d\theta_i=-\left.\omega\right|_{U_i}.$
Therefore on the intersections $U_{ij}$, we can find $s_{ij}\colon U_i\cap U_j\to \R$ such that
$ds_{ij}=\theta_j-\theta_i.$
Then, the cohomological obstruction class in $\check{H}^2(M;\R)$ is represented by the \v{C}ech cocycle 
$$e_{
ijk}=s_{ij}+s_{jk}-s_{ik}.$$
If $\mathscr{C}_{dR}\colon \check{H}^2(M;\R) \to H^2_{dR}(M)$ denotes the natural isomorphism from the \v{C}ech to the de Rham cohomology, then we have $\mathscr{C}_{dR}(e^A)=[\omega]$.
This example belongs to the, more general, class of algebroids for which the monodromy map $\partial^{A_\omega}\colon \pi_2(M)\to\R$ is explicit. Crainic and Fernandes proved that it is given by:
$$ \partial^{A_\omega}([\gamma])=-\int_{S^2}\gamma^*\omega,$$
for all $\gamma\colon S^2\to M$. Hence, the pairing between cohomology and homotopy identifies the Mackenzie class with the Crainic--Fernandes monodromy, i.e.
$$ \langle \mathscr{C}_{dR}(e^A), [\gamma] \rangle= \langle [\omega], [\gamma] \rangle=\int_{S^2}\gamma^*\omega=-\partial^{A_\omega}([\gamma]).$$
This example shows how the integrability problem, for algebroids coming from symplectic manifolds, is essentially equivalent to the prequantization problem in the approach of Weil \cite[Theorem 8.1.3]{MACKENZIE2000445}.

We can now state our result.
\begin{thm}\label{th1}
Let $A \Rightarrow M$ be a transitive Lie algebroid.
Fix a point $m \in M$ and let $\g=L_m$ be the adjoint Lie algebra at $m$ with simply connected associated Lie group $\widetilde{G}$.
The Crainic--Fernandes monodromy map $\partial^{\,A}: \pi_2(M,m) \to Z\widetilde{G}$ and the Meinrenken map
at $m$
agree. The Mackenzie class $e^A\in \check{H}^2(M;Z\widetilde{G})$
and the monodromy map are related via the natural isomorphism $\mathscr{C}:\check{H}^2(M;Z\widetilde{G}) \to H^2(M;Z\widetilde{G})$ 
and the natural pairing between singular cohomology and homotopy by the identity:
 $$\langle \mathscr{C}(e^A), \,\cdot \, \rangle =-\partial^A.$$
\end{thm}
The minus sign is only a matter of conventions, we have decided to keep the original choices in \cite{zbMATH00822678,zbMATH02057953} to help the direct comparison. \smallskip

The proof of Theorem \ref{th1} proceeds by reducing the problem to the case of the sphere; both the monodromy map and the Mackenzie class are natural (with respect to pullbacks of Lie algebroids). This allows us to deduce the general case by pulling back along maps $u\colon S^2\to M$. On the other hand, on the sphere, we identify the pairing with the Mackenzie class with the Meinrenken's version of the monodromy map by an explicit computation, then we identify the Meinrenken and Crainic--Fernandes monodromy maps.
In both cases the main instrument is the direct description of the group of gauge transformations of a trivial algebroid
in terms of the monodromy of the associated Maurer--Cartan forms \cite{zbMATH07480966}. \smallskip

Finally we note that it would be of interest to study the role of integrability obstructions in the globalization problem of the local Lie groupoid integrating a Lie algebroid \cite{zbMATH07176241,zbMATH01724481,zbMATH07453483}.
\smallskip 

We conclude by briefly describing the structure of the paper. In Section $2$ we recall some basic facts on Lie groupoids and algebroids with special care on the gauge transformations. Section $3$ is devoted to the construction of the Mackenzie's cohomological obstruction. Sections $4$ and $5$ deal with the monodromy maps of Crainic--Fernandes and Meinrenken, respectively. Theorem \ref{th1} is proved in section $6$ (in two steps). In the appendix we describe the isomorphism between \v{C}ech and singular cohomology using a \v{C}ech--de Rham type bicomplex and we collect some informations about Darboux derivatives.

\subsubsection*{Acknowledgement} Alessio Giannotta was supported by a scholarship financed by Ministerial Decree no. 118/2023 under the NRRP (NextGenerationEU, Mission 4) and by a Zegna Foundation scholarship supporting a research stay at Université Paris Cité.
Paolo Antonini wishes to thank Ugo Bruzzo, Domenico Fiorenza and Andrea Maffei for many helpful discussions about cohomology and  Eckhard Meinrenken for pointing to his lecture notes about transitive algebroids \cite{Meinnotes}.
Both the authors thank Iakovos Androulidakis and Georges Skandalis for many inspiring conversations.

\subsubsection*{Conventions and notations}
In this paper we adopt the convention that the Lie bracket 
$[X,Y]$ for $X,Y$ in the Lie algebra $\g$ 
of a Lie group $G$
is defined by the right invariant vector fields. This is quite common in the community of people working with groupoids. The notation $\operatorname{ad}_XY=[X,Y]$ still denotes the adjoint representation. \smallskip

Let 
 $\mathfrak{g}$ be a Lie algebra and $M$ be a manifold; for a form $\omega\in \Omega^1(M;\mathfrak{g})$ we put $[\omega,\omega](X,Y)=[\omega(X),\omega(Y)]$.
In particular $\omega$ is a Maurer--Cartan form if $$d\omega+[\omega,\omega]=0,$$ i.e. 
\[X\omega(Y)-Y\omega(X)-\omega([X,Y])-[\omega(X),\omega(Y)]=0
\]
for every $X,Y \in \Gamma(TM).$ \smallskip

On a Lie group $G$, the basic example of a Maurer--Cartan form is the right invariant Maurer--Cartan form 
 $\theta^{R} \in \Omega^1(G,\mathfrak{g})$; the unique right invariant $\mathfrak{g}$-valued form such that $\theta^R_{1}=\operatorname{Id}$. This can be characterized as the unique form such that $\theta^R(X^R)=X$ for every $X\in \mathfrak{g}$. \smallskip 

The Darboux derivative of a smooth function (see the appendix \ref{darboux}) $f:M \to G$ will be denoted with 
$\Delta(f):=f^*\theta^R \in  \Omega^1(M;\g)$.

\section{Background on groupoids and algebroids}
In this section we recall basic informations about Lie groupoids and (transitive) algebroids.
Trying to be quite concise here, we refer to 
\cite{zbMATH00822678,Meinnotes,zbMATH07480966} for the details.
\subsubsection*{Lie groupoids} Let $M$ be a manifold. A Lie groupoid $\mathcal{G}$ over $M$, denoted by $\mathcal{G}\rightrightarrows M $, is a (possibly non-Hausdorff) manifold of \textit{arrows} endowed with two surjective submersions $s,r\colon \mathcal{G}\to M$ called the \textit{source} and \textit{range} maps, an embedding $1\colon M \to \mathcal{G}$ called \textit{unit map} and a multiplication defined on the submanifold of composable arrows $\mathcal{G}^{(2)}=\{(g_2,g_1)\in\mathcal{G}^2\mid \, r(g_1)=s(g_2)\}$. These structure maps satisfy:
\begin{itemize}
\item \textit{associativity}: $(g_3g_2)g_1=g_3(g_2g_1)$;
\item \textit{units}: $1_{r(g)}g=g1_{s(g)}=g$;
\item \textit{inverse}: for all $g\in\mathcal{G}$ there exists an arrow $g^{-1}\in \mathcal{G}$ such that $gg^{-1}=1_{r(g)}$ and $g^{-1}g=1_{s(g)}$.
\end{itemize}
Fix a point $m\in M$. We denote by $\mathcal{G}_m=s^{-1}(m)$ the $s$-fibre of $m$, and by $\mathcal{G}_m^m=s^{-1}(m)\cap r^{-1}(m)$ the isotropy Lie group of $\mathcal{G}$ at $m$. We say that $\mathcal{G}$ is $s$-(simply) connected if, for every $m\in M$, the $s$-fiber $\mathcal{G}_m$ is (simply) connected.

Typical examples of Lie groupoids include Lie groups viewed as Lie groupoids over a point, the pair groupoid $M\times M$ over the manifold $M$, the action groupoid associated with a Lie group action on $M$ and the holonomy groupoid of a foliation on $M$.

By identifying the base manifold with the units submanifold, a morphism between the groupoids $\mathcal{G}\rightrightarrows M $ and $\mathcal{H}\rightrightarrows N $ is a smooth map $F\colon \mathcal{G} \to \mathcal{H}$ such that for all $(g_2,g_1)\in \mathcal{G}^{(2)}$ we have $(F(g_2),F(g_1))\in \mathcal{H}^{(2)}$ and $F(g_2g_1)=F(g_2)F(g_1)$ . 

\smallskip 
As in the case of Lie groups, Lie groupoids also admit an associated infinitesimal object. Let $\mathcal{G}$ be a Lie groupoid over $M$. The \textit{Lie algebroid of} $\mathcal{G}$ is the vector bundle $A(\mathcal{G}):=\left.\operatorname{ker}(ds)\right|_{1_M}$ over $M$, given by the restriction of $\operatorname{ker}(ds)$ on the units together with a bundle map $\sharp=dr\colon A(\mathcal{G})\to TM$, called the \textit{anchor}, and the Lie algebra bracket $[\cdot,\cdot]$  on $\Gamma(A(\mathcal{G}))$ induced from the Lie bracket on $\Gamma(T\mathcal{G})$.
\begin{exm}\label{extrivial}
$M$ be a manifold and let $G$ be a Lie group. The \textit{trivial} Lie groupoid over $M$ with structure group $G$ is the manifold $M\times M\times G$ with source and range maps given by the projection onto the second and first component, respectively. The multiplication is defined as follows:
$$(z,y,g)(y,x,h)=(z,x,gh).$$
Every automorphism $F$ of $M\times M\times G$ is of the form 
$$F(y,x,g)=(y,x,f(y)s(g)f(x)^{-1}),$$
where $f\colon M\to G$ is a smooth map and $s\colon G\to G$ is an automorphism of $G$ \cite[Example 1.2.5]{zbMATH00822678}. The pair $(f,s)$ is uniquely determined after fixing a base point $m\in M$ and requiring $f(m)=1$.

The Lie algebroid of $M\times M\times G$ is the vector bundle $TM\times \mathfrak{g}:=TM\oplus (M\times \mathfrak{g})$, where $\mathfrak{g}$ is the Lie algebra of $G$, with anchor map given by the projection onto $TM$ and bracket:
$$[X\oplus \xi,Y \oplus \eta]=[X,Y]\oplus ([\xi,\eta]+X(\eta)-Y(\xi))$$
 for $X,Y\in \Gamma(TM)$ and $\xi,\eta \in C^{\infty}(M,\g).$ 
\end{exm}
\subsubsection*{Algebroids}Generalizing the algebroid of a Lie groupoid, 
 an algebroid $A$ over $M$, denoted with $A \Rightarrow M$ is given by a vector bundle $A \to M$
with a Lie algebra structure on the sections of $A$ and 
a vector bundle map $\sharp:A \to TM$, the anchor, such that
$$[\xi,f\eta]=f[\xi,\eta]+(\sharp(\xi)f) \eta,\quad \xi,\eta \in \Gamma(A),  \quad f\in C^{\infty}(M).$$ 

Automatically $\sharp$ induces a Lie algebra morphism $\Gamma(A)\to \Gamma(TM).$
\smallskip 

Lie algebroids are ubiquitus in differential geometry: examples include Lie algebras, tangent bundles of manifolds, foliations, algebroids associated to principal bundles, (infinitesimal) Lie group actions and Poisson structures.

An algebroid is transitive when the anchor is surjective. In other words there is an extension 
 \begin{equation}
	\label{ext}
	\xymatrix{L \ar@{^{(}->}^{}[r]&A\ar@{>>}[r]^-{\sharp}&TM}\end{equation} where 
the kernel of the anchor map, $L$ is a Lie algebra bundle called the {\em{adjoint bundle}} of the algebroid. The fibre $L_m$ is called \textit{isotropy} Lie algebra at $m$.
\subsubsection*{Products, pullbacks, morphisms}
The category of Lie algebroids admits products and morphisms. The example \ref{extrivial} corresponds to the product of a Lie algebra $\g \Rightarrow \bullet$ with $TM$ which serves as a local model for any transitive algebroid. We will refer to it as a \textit{trivial} Lie algebroid.
 
 \smallskip
 
Morphisms and pullbacks of algebroids require some care to be defined. 
Morphisms of algebroids over the same base are defined in a straightforward way as vector bundle maps preserving the anchor and inducing Lie algebra morphisms on sections. On the other hand, morphisms of algebroids over different bases will be described after introducing pullbacks.

\smallskip 
Pullbacks of algebroids exist under suitable conditions  \cite{zbMATH00822678,zbMATH08133560}. We will only consider pullbacks of transitive Lie algebroids; for which this construction is always defined. In the following we shall focus on them.

So let $A \Rightarrow M$ be a transitive algebroid and $f:B \to M$ a smooth map. 
First form the pullback square:
\[\begin{tikzcd}
	{f^{!}A} & {f^*A} \\
	TB & {f^*TM}
	\arrow[from=1-1, to=1-2]
	\arrow[from=1-1, to=2-1]
	\arrow["{f^*\sharp}", from=1-2, to=2-2]
	\arrow["df"', from=2-1, to=2-2]
\end{tikzcd}\]
as vector bundles over $B$. 
Thus $f^{!}A=TB \times_{f^*TM} f^*A \subseteq TB \oplus f^*A$ is the fibered product, a smooth subbundle of $TB \oplus f^*A$.

Sections of $f^{!}A$ are 
locally generated over $C^{\infty}(B)$ by
pairs $(X,f^*\xi)$ with $X\in \Gamma(TB)$ and $\xi \in \Gamma(A)$ such that $df(X)=\sharp(\xi) \circ f$.
The anchor of $f^{!}A$ is the projection on $TB$ and the bracket is specified by the requirement
of being compatible with the second projection $f^{!}A \to f^*A$ and the anchor.

In particular it is determined by the formula:
$$[(X,hf^*\xi),(Y,gf^*\eta)]=([X,Y],hg f^*[\xi,\eta]+X(g)f^*\eta-Y(h)f^*\xi),$$
for $X,Y\in \Gamma(TB)$ and $\xi,\eta \in \Gamma(A)$ such that $df(X)=\sharp(\xi) \circ f$ and $df(Y)=\sharp(\eta) \circ f$.

The pullback of a transitive algebroid is transitive, with adjoint bundle $f^*L$ where $L$ is the adjoint bundle of $A$.\smallskip

When the pullback algebroid $f^{!}A$ is defined,
morphisms of algebroids over different bases are easier to describe. 
Given algebroids $A_1 \Rightarrow B$ and $A_2 \Rightarrow M$ 
with $A_2$ transitive, a morphism  $A_1 \to A_2$ 
is a vector bundle map $\Phi:A_1 \to A_2$ covering a smooth map $f:B \to M$ such that:
\begin{itemize}
\item it is anchor preserving:$\sharp_2 \circ \Phi = df \circ \sharp_1$,
\item the induced map $A_1 \to f^{!}A_2, \,\, \xi \mapsto (\sharp_1(\xi),\Phi(\xi))$ is a morphism of algebroids over $B$.
\end{itemize} 

\subsubsection*{The integrability problem} 
As we saw at the beginning of this section, every Lie groupoid has an associated Lie algebroid. This defines the Lie functor $\mathcal{G}\mapsto A(\mathcal{G})$ from the category of Lie groupoids to the category of Lie algebroids \cite[Proposition 3.5.10]{zbMATH00822678}.  
A Lie algebroid $A$ over $M$ is \textit{integrable} if there exists a Lie groupoid $\mathcal{G}$ such that $A= A(\mathcal{G})$ up to isomorphism. When integrable, a Lie algebroid is always integrated to an $s$-simply connected Lie groupoid \cite[Proposition 3.3]{cfddeee2-66aa-333f-9ba9-b0e0e3e227c3}, which is unique up to isomorphism, as follows by the Lie II theorem for Lie algebroids.
\begin{thm}\em{ \cite[Theorem A.1]{MACKENZIE2000445}}\label{lie2} Let $\mathcal{G}$ and $\mathcal{H}$ be Lie groupoids over $M$ and $N$, respectively. Suppose that $\mathcal{G}$ is $s$-simply connected and let $\Psi\colon A(\mathcal{G})\to A(\mathcal{H})$ be a Lie algebroid morphism. Then there exists a unique Lie groupoid morphism $F\colon \mathcal{G}\to \mathcal{H}$ such that $A(F)=\Psi$.
\end{thm}

In the transitive case the integrability problem can be phrased in the world of principal bundles.

Recall that a Lie groupoid is transitive when it has one orbit i.e. for any point $m\in M$ the map $r:\mathcal{G}_m:=s^{-1}(m) \to M$ is surjective. Then $\mathcal{G}_m$ is a (locally trivial) $G=\mathcal{G}_m^m$-principal bundle. The associated Atiyah sequence (see \cite{zbMATH00822678,zbMATH07480966}) defines a transitive algebroid on $M$.
On the other hand from a principal bundle, one has an associated transitive groupoid, its {\em{Gauge groupoid}}.

Start with an integrable transitive Lie algebroid $A \Rightarrow M$; any integrating groupoid is transitive and any $s$-fiber 
defines a principal bundle whose Atiyah sequence recovers $A$.

However this correspondence involves the choice of basepoints
and, as explained by Mackenzie \cite[Chapter 1]{zbMATH00822678} it is prefeable not to identify 
the category of principal bundles with that of transitive Lie groupoids.

\subsubsection*{Flat connections and isomorphisms}
Let $A \Rightarrow M$ be a transitive algebroid.
A {\em{connection}} $\Theta: TM \to A$ on $A$ is a splitting  
of the sequence \eqref{ext}
as vector bundles. Associated with $\Theta$ there is a curvature form $\overline{R}_{\Theta}\in \Omega^2(M,L)$ defined by
$$\overline{R}_{\Theta}(X,Y):=
\Theta[X,Y]-
[\Theta X,\Theta Y],$$
for $X,Y \in \Gamma(TM)$.
The curvature measures the failure of the connection of being a splitting of Lie algebroids in which case it is called {\em{flat}}.

We use flat connections to describe isomorphisms of algebroids $\Psi:TM \times \mathfrak{g} \longrightarrow A$.
Indeed $\Psi$ breaks into two pieces:

\begin{itemize}
\item 
a trivialization of Lie algebra bundles
$\psi:M \times \mathfrak{g}\longrightarrow L$, 
\item a flat connection $\Theta: TM \longrightarrow A$
\end{itemize}
in a way that $\Psi(X,\xi)=\Theta(X)+\psi(\xi)$.

These two objects obey a compatibility condition: $\psi(X(\xi))=[\Theta (X),\psi (\eta)]$ for every $X\in \Gamma(TM)$ and $\xi\in C^{\infty}(M,\mathfrak{g})$.
Vice versa a compatible pair $(\Theta,\psi)$ gives rise to an isomorphism $TM \times \g \to A$.

\subsubsection*{Gauge transformations}
Let $TM \times \mathfrak{g}$ be a trivial Lie algebroid. The group of gauge transformations $\operatorname{Gauge}(TM \times \mathfrak{g})$ consists of all Lie algebroid automorphisms of $TM \times \mathfrak{g}$ covering the identity on $M$.

 For the trivial Lie algebroid $TM\times \mathfrak{g}$, a connection $\Theta$ is given by
 $$ \Theta(X)=X\oplus \theta(X),$$
 where $\theta\in\Omega^1(M,\mathfrak{g})$ and a trivialization $\psi\colon M\times \mathfrak{g}\longrightarrow L$ has the form
 $$ \psi(m,\xi)=0\oplus \Phi_m(\xi), $$
for some $\Phi\in C^{\infty} (M,\operatorname{Aut}(\mathfrak{g}))$.

 Therefore, $\operatorname{Gauge}(TM \times \mathfrak{g})$ can be identified with the set of pairs $(\theta,\Phi)$ where $\theta\in\Omega^1(M,\mathfrak{g})$ is a Maurer--Cartan form, i.e.
\footnote{meaning exactly that $\Theta:TM \to TM\times \mathfrak{g}$ is a flat connection, i.e. a Lie algebroid morphism}

\begin{equation}\label{comp1}d\theta +[\theta,\theta]=0,\end{equation}
 and $\Phi \in C^{\infty}(M,\operatorname{Aut}(\mathfrak{g}))$ satisfies the compatibility condition
 \begin{equation}\label{comp2}
d\Phi+\operatorname{ad}_{\theta} \circ \, \Phi=0.
\end{equation}
Explicitly, this reads 
 \[ 
d\Phi(X)(\xi)+[\theta(X),\Phi(\xi)]=0, \quad X\in TM, \, \xi \in \mathfrak{g}.
\]
The associated automorphism of $TM\times \mathfrak{g}$ is given pointwisely by  
$$ \Psi(X\oplus \xi)=X\oplus \theta(X)+\Phi(\xi). $$

\subsubsection*{Framings} Following \cite{zbMATH07480966}, a  framing for $A \Rightarrow M$ at $m\in M$ is the choice of an isomorphism $\mathfrak{g} \to \iota^{!}_m A=L_m $ where $\iota_m$ is the inclusion of $m$ in $M$. Any frame can be extended to an isomorphism of algebroids $TU \times \mathfrak{g}  \to A\rest{U}$
(restricting to the given framing)
over any open contractible neighborhood of $m$. The extension is unique up to smooth isotopy \cite[Proposition 3.6]{zbMATH07480966}.
\smallskip 

The following expresses the gauge group in terms of the universal cover of $M$ and the group $\widetilde{G}$. 
\begin{thm}\label{gauge1}\em{(\cite[Proposition 3.7]{zbMATH07480966})}\label{gauge1}
Let $m \in M$ be a reference point and let $\widetilde{M}$ be the universal covering of $M$ with projection $\pi: \widetilde{M}\to M$, fixed point $\widetilde{m}$ above $m$ and right action of $\pi_1(M,m)$. 

There is a $1 : 1$ correspondence between $\operatorname{Gauge}(TM \times \mathfrak{g})$ and the set of triples $(\rho,f,\psi)$ such that 
\begin{itemize}
\item $\rho: \pi_1(M,m) \to Z\widetilde{G}$ is a representation;
\item $f: \widetilde{M} \to \widetilde{G}$ is a smooth quasi-periodic smooth map with $f(\widetilde{m})=1$, i.e. $f(\widetilde x \cdot \gamma)=f(\widetilde x)\rho(\gamma)$ for $\widetilde x\in \widetilde{M}$ and $\gamma \in \pi_1(M,m)$,
\item $\psi \in \operatorname{Aut}(\mathfrak{g})$.
\end{itemize}
The triple $(\rho,f,\psi)$ defines a gauge transformation of $\pi^{!}(TM \times \g)\cong T\widetilde{M}\times \g$,
$$\widetilde{X} \oplus \xi \longmapsto \widetilde{X}\oplus \Delta(f)(\widetilde{X}) +(\operatorname{Ad}_f \circ \,\psi)\,  \xi, $$
which descends to an element of $\operatorname{Gauge}(TM \times \mathfrak{g})$.
Moreover, after fixing a frame at $m$, the automorphisms preserving the frame are exactly those with $\psi=\operatorname{Id}$.
\end{thm}

\begin{proof}
This was proved by Meinrenken; we write a slightly different proof based on theorem \ref{lie2}. Let $\Psi$ be the automorphism of $TM \times \g$ associated with a pair  $(\theta,\Phi)$ satisfying 
\eqref{comp1} and \eqref{comp2}, and given by
\begin{equation}\label{automorphismtrivial}
\Psi(X\oplus \xi)=X\oplus \theta(X)+\Phi(\xi).
\end{equation}
The $s$-simply connected groupoid 
integrating $TM \times \g$ is $\mathcal{G}:=\dfrac{\widetilde{M}\times \widetilde{M}}{\pi_1(M,m)} \times \widetilde{G}$ and $\Psi$ integrates to a Lie groupoid automorphism $F$ of $\mathcal{G}$. There exists a unique lift $\widetilde F$ to an automorphism of the trivial groupoid $\widetilde M\times \widetilde M \times \widetilde G $ making the diagram:
\begin{equation}\label{commdiagr}
\begin{tikzcd}
	{\widetilde M\times \widetilde M \times \widetilde G} & {\widetilde M\times \widetilde M \times \widetilde G} \\
	{\dfrac{\widetilde{M}\times \widetilde{M}}{\pi_1(M,m)}\times \widetilde{G}} & {\dfrac{\widetilde{M}\times \widetilde{M}}{\pi_1(M,m)}\times \widetilde{G}}
	\arrow["\widetilde{F}", from=1-1, to=1-2]
	\arrow["q\times \operatorname{Id}_{\widetilde{G}}", from=1-1, to=2-1]
	\arrow["q\times \operatorname{Id}_{\widetilde{G}}", from=1-2, to=2-2]
	\arrow["F", from=2-1, to=2-2]
\end{tikzcd}
\end{equation}
commute.

By example \ref{extrivial}, there exists a unique pair $(f,s)$ with $f\colon \widetilde{M}\to \widetilde{G}$, $f(\widetilde{m})=1$, and $s\in \operatorname{Aut}(\widetilde{G})$ such that
\begin{equation*}
\widetilde{F}(\widetilde{y},\widetilde{x},g)=(\widetilde{y},\widetilde{x},f(\widetilde{y})s(g)f(\widetilde{x})^{-1}).
\end{equation*}

For $\gamma\in \pi_1(M,m)$, the identity $F([\widetilde{y},\widetilde{x}],1)=F([\widetilde{y}\gamma,\widetilde{x}\gamma],1)$ implies
$f(\widetilde x)^{-1}f(\widetilde x\gamma)=f(\widetilde y)^{-1}f(\widetilde y\gamma)$, for all $\widetilde x, \widetilde y\in \widetilde{M}$. Thus, there exists a map $\rho\colon \pi_1(M,m)\to \widetilde{G}$ such that $f(\widetilde{x}\gamma)=f(\widetilde{x})\rho(\gamma)$. Since $F$ is a groupoid morphism, $\rho$ is a group homomorphism. Furthermore, the identity $F([\widetilde{y},\widetilde{x}],g)=F([\widetilde{y}\gamma,\widetilde{x}\gamma],g)$, implies that $\rho(\pi_1(M,x_0))\subseteq Z\widetilde{G}$.

By the diagram \eqref{commdiagr}, the automorphism $A(\widetilde{F})$ of $T\widetilde{M}\times \g$ induced from $\widetilde{F}$ is:
\begin{equation}
A(\widetilde{F})(\widetilde{X} \oplus \xi) = \widetilde{X}\oplus \Delta(f)(\widetilde{X}) +(\operatorname{Ad}_f \circ \,\psi)\,  \xi 
\end{equation}
where $\psi\in \operatorname{Aut}(\mathfrak{g})$ is the Lie algebra automorphism associated with $s$. Since $\widetilde{F}_*$ descends to $\Psi$, we obtain:
\begin{equation}\label{Conditionspullbackgauge}
\Delta(f)=q^*\theta,\quad \operatorname{Ad}_f \circ\,\psi=\Phi.
\end{equation}

Conversely, any triple $(\rho,f,\psi)$ uniquely determines a Lie groupoid automorphism of $\widetilde M\times \widetilde M \times \widetilde G$ descending to an automorphism of $\mathcal{G}$ and hence induces an automorphism of the Lie algebroid $TM\times\mathfrak{g}$.
\end{proof}

\section{The Mackenzie class}

\subsubsection*{Transition datas}
The starting point of the construction of the Mackenzie obstruction class is the local triviality of any transitive algebroid \cite[Theorem 8.2.1]{zbMATH00822678}.
A short proof using flows is given in \cite[Proposition 3.6]{zbMATH07480966}. \smallskip 

We can thus find an open contractible cover $\{U_i\}_i$ of $M$ together with Lie algebroid isomorphisms
(of algebroids)
$\Psi_i\:TU_i  \times \mathfrak{g} \longrightarrow A\rest{U_i}$. 
Decomposing these 
into trivializations of Lie algebra bundles $\psi_i$ and flat connections $\Theta_i$,
we get a collection $\{U_i,\psi_i,\Theta_i\}_i$ called a {\emph{Lie algebroid atlas}} for $A$. \smallskip 

We examine the behavior of these local isomorphisms on non empty intersections.
\noindent Let $a_{ij}:U_{ij} \longrightarrow \operatorname{Aut}(\mathfrak{g})$ be the associated transition functions: $a_{ij}=\psi_i^{-1}\circ\, \psi_j$; then we may define
 connection forms
$\ell_{ij}\in \Omega^1(U_{ij},L)$ and
$\chi_{ij}\in \Omega^1(U_{ij},\mathfrak{g})$ according to
$$ \ell_{ij}=\Theta_j-\Theta_i, \quad \chi_{ij}:=\psi_{i}^{-1}\circ \ell_{ij}.$$

These determine the transition map:
$${\Psi}_i^{-1} \circ \left( {\Psi}_j \rest{{TU_{ij}}\times \g}\right): TU_{ij} \times \g \longrightarrow TU_{ij} \times \g,$$
$$X\oplus\xi \longmapsto X\oplus{\chi}_{ij}(X)+ {a}_{ij}(\xi).$$

\begin{prop}\cite[Theorem 8.2.4]{zbMATH00822678}. \label{transition forms}
The forms $\chi_{ij}$ satisfy:
\begin{enumerate}\label{transition forms}
\item $d\chi_{ij}+[\chi_{ij},\chi_{ij}]=0$ i.e. each form $\chi_{ij}$ is a Maurer--Cartan form,
\item $\chi_{ik}=\chi_{ij}+a_{ij}(\chi_{jk})$ on $U_{ijk}$,
\item $\Delta(a_{ij})=\operatorname{ad}\circ \chi_{ij}$ on $U_{ij}$ with respect to the right Darboux derivative $\Delta$ for the group $\operatorname{Aut}(\mathfrak{g})$.
\end{enumerate}
\end{prop}

\begin{defn}
A \emph{system of transition datas} $\{U_i,a_{ij},\chi_{ij}\}_i$ \emph{for a transitive algebroid} is given by an open cover $\{U_i\}$, a cocycle $\big{\{}a_{ij}:U_{ij} \longrightarrow \operatorname{Aut}(\mathfrak{g})\big{\}}_i$ and forms $\chi_{ij}$ satisfying the properties $(1),(2)$ and $(3)$ in Proposition \ref{transition forms}.
\end{defn}
A system of transition datas completely recovers $A$.It is then natural to ask under which condition these systems can be integrated to a Lie groupoid.

\subsubsection*{Transition datas from a groupoid}\label{derivation}
Start with a transitive Lie groupoid $\mathcal{G}\rightrightarrows  M$ with Lie algebroid $A=A(\mathcal{G})$. We get a system of transition datas applying the Lie functor.
\smallskip 

 Fix a reference point $m\in M$ and consider the principal bunde $r\colon \mathcal{G}_m\to M$ with structure group $G:= \mathcal{G}_m^m$.
\begin{defn}
A section atlas consists of a collection $\{U_i, \sigma_i\}_i$ where $\{U_i\}$ is an open cover and each $\sigma_i\colon U_i\to \mathcal{G}_m$ is a section of the range map.
\end{defn} 
Given a section atlas we obtain a $G$-valued cocycle
$\sigma_{ij}(x):=\sigma_i(x)^{-1}\sigma_j(x)$ for $x\in U_{ij}$. The following proposition although immediate, suggests the right choice of the ingredients
for the construction of the obstruction class.
\begin{prop}\label{propcocycle}
A section atlas $\{U_i,\sigma_i\}_i$ of a transitive groupoid $\mathcal{G}$ induces a 
family of trivializations for $A(\mathcal{G})$ whose associated 
system of transitions datas $\{U_i,a_{ij},\chi_{ij}\}_i$ has the properties:
\begin{enumerate}
\item $a_{ij}=\operatorname{Ad} \circ \,\sigma_{ij}$ on $U_{ij}$ with respect to the adjoint representation $\operatorname{Ad}:G \to \operatorname{Aut}(\mathfrak{g})$.
\item $\Delta(\sigma_{ij})=\chi_{ij}$ on $U_{ij}$.
\end{enumerate}
\end{prop}
\begin{proof}
\noindent The section atlas induces a family of groupoid trivializations:
\begin{equation*}
\Sigma_i:U_i \times U_i \times G \longrightarrow \mathcal{G}_{U_i}^{U_i}, \quad \Sigma_i(y,x,g)=\sigma_i(y)\cdot g \cdot \sigma_i(x)^{-1}.\end{equation*}
On the local pieces we have obvious split sequences of groupoids:
\begin{equation*}
\xymatrix@1{U_i \times G\,\,\ar@{^{(}->}[r]^-{\eta_i} & U_i \times U_i \times G \ar@{->>}[r]& U_i \times U_i}
\end{equation*}
where:
$\eta_i(x,g)=(x,x,g)$ and
splitting $\theta_i(y,x)=(y,x,1).$
Applying the Lie functor to each of these arrows we get a Lie algebroid atlas 
$(U_i,\psi_i,\Theta_i)_i$
for $A$ where
$$\Theta_i=A(\Sigma_i \circ \theta_i), \quad \psi_i=A(\Sigma_i\circ \eta_i).$$
The system of transition datas $(U_i,a_{ij},\chi_{ij})_i$ associated to $(U_i,\psi_i,\Theta_i)_i$ 
is computed by differentiation of the transition automorphisms 
\[
\begin{array}{rcl}
\Sigma_{ij}=\Sigma_i^{-1} \circ
\bigl(
  \Sigma_j \rest{U_{ij} \times U_{ij} \times G}
\bigr)
& : &
U_{ij} \times U_{ij} \times G
\longrightarrow
U_{ij} \times U_{ij} \times G \\[0.6em]
& &
(y,x,g)
\longmapsto
\bigl(
  y,\;
  x,\;
  \sigma_{ij}(y)\, g \, \sigma_{ij}(x)^{-1}
\bigr).
\end{array}
\]
Therefore, by differentiation, the induced map at the algebroid level is: 
$$A(\Sigma_{ij})(X\oplus\xi)=X\oplus\Delta(\sigma_{ij})(X)+\operatorname{Ad}_{\sigma_{ij}(x)}\xi.$$
\end{proof}

\subsection{The cohomological obstruction}
Let $A\Rightarrow M$ be a transitive algebroid and $p:\widetilde{M}\to M$ be the universal covering of $M$.
Mackenzie proved that $A$ is integrable if and only if the pullback algebroid $p^{!}A \Rightarrow \widetilde{M}$ is 
\cite[Theorem 8.3.4]{zbMATH00822678}.
Therefore the cohomological obstruction to the integrability  is first defined for an algebroid on a simply connected base, then in the general case, one just defines it as the one of the lift to the universal cover \cite[Definition 8.3.5]{zbMATH00822678}. \smallskip 

The Mackenzie obstruction arises from reversing the differentiation process starting from a trivializing family of sections of a transitive groupoid $\mathcal{G}$ and recovering a system of transition datas for $A=A(\mathcal{G})$. 
So begin with a transitive algebroid $A\Rightarrow M$ over a simply connected base and fix a point $m\in M$; let $\mathfrak{g}=L_m$ be the isotropy Lie algebra at $m$.
If $A$ comes from a Lie groupoid $\mathcal{G}$ then it necessarily admits an algebroid atlas with transition datas
with $a_{ij}$ and $\xi_{ij}$
 satisfying 
the conditions $(1)$ and $(2)$ in the Proposition \ref{propcocycle}. \smallskip 

Recall that the adjoint group of $\g$, denoted with $\operatorname{Ad}(\g)$ is the
unique connected Lie subgroup of $\operatorname{Aut}(\g)$ whose Lie algebra is the image of the adjoint representation $\g \to \mathfrak{gl}(\g)$. For every connected Lie group $G$ integrating $\g$ we have $\operatorname{Ad}(\g)=\operatorname{Ad}(G)$, the image of the adjoint representation of $G$. This group does not depends on the 
Lie group integrating $\g$; we can take the simply connected one $\widetilde{G}$. 

Therefore, if the $s$-fibers of $\mathcal{G}$ are connected, the isotropy group $G=\mathcal{G}_m^m$ is connected too, and
$\operatorname{Ad}(G)=\operatorname{Ad}(\g)$.
The first compatibility condition in Proposition  \ref{propcocycle} implies that the cocycle takes values in $\operatorname{Ad}(G)$. \smallskip

The effect is that if we want to try to revert the differentiation process we must begin with special algebroid atlases.
However this does not constitute an obstruction. Mackenzie showed that for any transitive algebroid we may always find an atlas 
 $\{U_i,\psi_i,\Theta_i\}_i$
 with the $a_{ij}$ taking values in $
 \operatorname{Ad}(\widetilde{G})=\operatorname{Ad}(\g)$. \smallskip

More precisely, associated with $L$ is the transitive Lie groupoid $\Phi_{\operatorname{Aut}}(L)$ whose arrows from $x$ to $y$ are the Lie algebra isomorphisms $L_x \to L_y$. Its algebroid is denoted with $\mathcal{D}_{\operatorname{der}}(L)$. A section $u \in \Gamma(\mathcal{D}_{\operatorname{der}}(L))$ is identified with a bracket derivation: a first order differential operator $u: \Gamma(L) \to \Gamma(L)$ such that 
$$u[\xi,\eta]=[u\xi,\eta]+[\xi,u\eta], \quad \xi,\eta \in \Gamma(L).$$
The adjoint representation of $A$ is a morphism of Lie algebroids $\operatorname{ad}: A \to \mathcal{D}_{\operatorname{der}}(L)$ with constant rank. In particular the image $\operatorname{ad}(A) \leq \mathcal{D}_{\operatorname{der}}(L)$ is a transitive subalgebroid\footnote{the definition of subalgebroid can be found in \cite{zbMATH00822678}, here we only need the existence of the transitive groupoid associated with the adjoint representation}. 
There corresponds an $s$-connected transitive Lie subgroupoid\footnote{in the sense of the image of an injective immersion} of $\Phi_{\operatorname{Aut}}(L)$, denoted with $\operatorname{Int}(A)$, the \emph{groupoid of inner automorphisms of} $A$. 

The Lie algebra of the isotropy group $\operatorname{Int}(A)_m^m$ can be identified with the image of the adjoint representation of $\g$ and, again ($M$ is simply connected) $\operatorname{Int}(A)_m^m$ 
is connected. It follows
$$\operatorname{Int}(A)_m^m=\operatorname{Ad}(\widetilde{G}).$$
Summing up: any section atlas of $\operatorname{Int}(A)$ (with fixed base point) gives a trivializing atlas $\psi_i: U_i\times \mathfrak{g} \longrightarrow L\rest{U_i}$ whose associated cocycle $a_{ij}$ is valued in $\operatorname{Ad}(\widetilde{G})$.\smallskip

To continue with the process of trying to integrate the two conditions in Proposition \ref{propcocycle} we have to complete the $\psi_i$'s to an algebroid atlas.

This second step involves the following Theorem by Mackenzie.
 We state it in a slightly different, equivalent way.

\begin{thm}\cite[Theorem 8.3.1]{zbMATH00822678} \label{macextens}
\label{completiontheorem}
With the above notations,
let $A$ be a transitive algebroid (on an arbitrary base $M$) and fix a section $ U \longrightarrow  \operatorname{Int}(A)_m$ on an open contractible set with associated trivialization $\psi:U \times \mathfrak{g} \longrightarrow L\rest{U}$. We can find a flat connection $\Theta: TU \to A$ such that the map
$\Theta + \psi: TU \oplus (U \times \mathfrak{g}) \longrightarrow A\rest{U}$ is an isomorphism of algebroids.
\end{thm}

\begin{rmk}
If $M$ is simply connected, $\g=L_m$ and $\psi: M \times \g \to L$ is a trivialization of Lie algebra bundles that extends to an algebroid trivialization as in Theorem \ref{macextens}, then necessarily $\psi$ comes from a section $U \to \operatorname{Int}(A)_m$. Indeed first define the groupoid morphism $\underline{\Psi}:M \times M \to {\Phi}_{\operatorname{Aut}}(L)$ with $\underline{\Psi}(y,x):=\psi_y \circ \psi_x^{-1},$ then equation \eqref{comp2} implies that the tangent map $A(\underline{\Psi})$ maps $TM$ to $\operatorname{ad}(A)\subseteq \mathcal{D}_{\operatorname{der}}(L)$ and conclude using Lie $\rm{II}$.
\end{rmk}

With this in mind we are ready to define the cohomological obstruction.
\subsubsection*{Recipe for the Mackenzie obstruction class:}
\begin{itemize}
\item fix a point $m\in M$ and let $\g:=L_m$ the isotropy Lie algebra with simply connected associated group $\widetilde{G},$
\item take any section atlas of $\operatorname{Int}(A)_m$ on a good cover 
 $\{U_i\}_i$; this gives trivializations $\psi_i: U_i \times \mathfrak{g}\to L\rest{U_i}$ and use Theorem \ref{completiontheorem} to complete these to an algebroid atlas $(U_i,\psi_i,\Theta_i)_i$.
Now the associated system of transition datas
     $(U_i,a_{ij},\chi_{ij})_i$ satisfies $a_{ij}:U_{ij}\rightarrow \operatorname{Ad}(\widetilde{G})$.
\item 
Integrate the $\chi_{ij}$ to find
functions $s_{ij}:U_{ij}\to \widetilde{G}$
such that $\Delta(s_{ij})=\chi_{ij}$. These are
unique up to right translation for an element of $\widetilde{G}$.
\item Computing the Darboux derivative with respect to the group $\operatorname{Ad}(\widetilde{G})$ we find
$$\Delta(\operatorname{Ad} \circ \,s_{ij})=\operatorname{ad} \circ \,\Delta(s_{ij})= \operatorname{ad}\circ \,\chi_{ij} =\Delta(a_{ij}).$$
The intermediate $\Delta$ is in $\widetilde{G}$ but the first and the last ones are in $\operatorname{Ad}(\widetilde{G})$. It follows that the $s_{ij}$ can be chosen such that 
$$a_{ij} = \operatorname{Ad}\circ \,s_{ij}$$
for every $i,j$.
\item For non empty $U_{ijk}$ define functions $e_{ijk}:U_{ijk}\longrightarrow \widetilde{G}$ by $e_{ijk}:=s_{jk}s_{ik}^{-1}s_{ij}$.
\end{itemize}
\begin{thm}{\em{(Mackenzie)}}
Let $A$ be a transitive algebroid on a simply connected manifold. The functions $e_{ijk}$ defined by the above recipe are valued in $Z\widetilde{G}$ and define a $2$-\u{C}ech cocycle
$(e_{ijk})_{ijk}.$
The corresponding cohomology class 
$$e^A:=[(e_{ijk})_{ijk}]\in \check{H}^2(M,Z\widetilde{G}) $$ does not depends on the choices.
The algebroid $A$ is integrable if and only if
$e$ belongs to the image of $\check{H}^2(M,D)$ for a discrete subgroup $D\subset Z\widetilde{G}$.
\end{thm}
 The class $e$ is the cohomological obstruction by Mackenzie.
\begin{proof}
The statement puts together \cite[Theorem 8.3.2]{zbMATH00822678} and the computations before it.
\end{proof}
When $M$ is non simply connected one simply
considers the universal covering $\pi:\widetilde{M}\to M$ and
defines the class of $A\Rightarrow M$ as the class of $\pi^*A$ in $\check{H}^2(\widetilde{M},Z\widetilde{G}).$

\begin{prop}\label{functorialityMackenzie}
Let $A$ be a transitive Lie algebroid over a simply connected manifold $M$ and let $f\colon B\to M$ be a smooth map with $B$ simply connected too. Then 
$$ e^{f^!A}=f^*(e^A),$$
where $f^*\colon \check{H}^2(M,Z\widetilde{G}) \to \check{H}^2(B,Z\widetilde{G}) $ is the homomorphism induced by $f$.
\end{prop}
\begin{proof} Let $(U_i,\psi_i,\Theta_i)_i$ be an algebroid atlas for $A$. Fix a good cover $\{V_i\}_i$ of $B$ such that $f(V_i)\subseteq U_{\lambda(i)}$, for a suitable refining map $\lambda$. The pullback of the trivializations $\psi_{\lambda(i)}$ and the flat connections $\Theta_{\lambda(i)}$ define an algebroid atlas $(V_i,\psi_i',\Theta_i')$ for $f^!A$. A direct computation shows that the associated system of transition data $(V_i,a_{ij}',\chi_{ij}')$ satisfies $\chi_{ij}'=\chi_{\lambda(i)\lambda(j)}\circ f$ and $a_{ij}'=a_{\lambda(i)\lambda(j)}\circ f$.
Accordingly, the corresponding solutions of $\Delta(s_{ij}')=\chi_{ij}'$ and $a_{ij}'=\operatorname{Ad}\circ \,s_{ij}'$ are obtained pulling back the solutions $s_{\lambda(i)\lambda(j)}$ associated with $A$, namely
$s_{ij}'=s_{\lambda(i)\lambda(j)}\circ f.$
It follows that the associated  \u{C}ech cocycles satisfy $e_{ijk}'=e_{\lambda(i)\lambda(j)\lambda(k)}\circ f$, and hence $e^{f^!A}=f^*(e^A).$
\end{proof}
\section{The Crainic--Fernandes monodromy map}
In this section, we briefly recall the construction of the monodromy map introduced by Crainic and Fernandes \cite{zbMATH02057953}. Their solution of the integrability problem for Lie algebroids is based on the extension of the path-space approach for Lie algebras to the algebroid setting. The fundamental objects leading to the integrability obstructions are $A-paths$ and $A-homotopies$.
\smallskip 

Let $\mathcal{G}$ be a Lie groupoid. A $\mathcal{G}-path$ is a path $g\colon I\to \mathcal{G}$ contained in a single $s$-fiber of $\mathcal{G}$ and starting at the identity of that fiber. If $A$ is the Lie algebroid of $\mathcal{G}$, the $s$-simply connected Lie groupoid integrating $A$, i.e. the monodromy groupoid of $\mathcal{G}$ is recovered as
\begin{equation}\label{MonG} \widetilde{\mathcal{G}}=P(\mathcal{G})/ \sim,\end{equation}
where $P(\mathcal{G})$ is the space of $\mathcal{G}$-paths endowed with the $C^2$-topology and $\sim$ is the equivalence relation given by $C^1$-homotopy with fixed endpoints. Groupoid multiplication is induced by the concatenation of paths. One of the most important points in the construction is the fact that $\widetilde{\mathcal{G}}$ can be described entirely in terms of infinitesimal datas via $A$-paths and $A$-homotopies.

\subsubsection*{A-homotopies}

Let $A\rightarrow M$ be a Lie algebroid. A $C^1$-path $a\colon I\to A$ is an $A-path$ if
\begin{equation*}
\#a(t)=\dot{\gamma}(t),
\end{equation*}
where $\gamma=\pi\circ a$ is the projection of $a$ on $M$, or, equivalently, if the bundle map
$a\, dt\colon TI\to A$
is a Lie algebroid morphism over $\gamma$.

If we start with a groupoid $\mathcal{G}$
integrating $A$,
and $g:I \to \mathcal{G}$ is a $\mathcal{G}$-path, its right differentiation:
$$D^R(g)(t)=\frac{d}{ds}\Big|_{s=0}g(s)\cdot g(t)^{-1},$$
is an $A$-path. Conversely, any $A$-path $a$ integrates to a unique $\mathcal{G}$-path $g$ such that $D^R(g)=a$. Indeed, when endowed with an appropriate Banach manifold structure
the map $D^R$ is an homeomorphism. Moreover $D^R$ turns fixed endpoints homotopy into a relation that can be described infinitesimally without knowing the existence of $\mathcal{G}$, this is exactly the notion of $A$-homotopy.
\smallskip 

More precisely, an $A-homotopy$ between two $A-paths$ $a_0,a_1\colon I\to A$ is a Lie algebroid morphism
\begin{equation*}
H \colon TI^2\to A
\end{equation*}
satisfying the boundary conditions:
\begin{equation*}
\left.H\right|_{T\left( \{0\}\times I \right)}=a_0, \quad \left.H\right|_{T \left( \{1\}\times I \right)}=a_1, \quad \left.H\right|_{T \left( I\times \{0\} \right)}=\left.H\right|_{T \left( I\times \{1\} \right)}=0.
\end{equation*}
Notice that $H$ covers a standard homotopy with fixed enpoints between the projections of $a_0$ and $a_1$ on $M$.
\begin{rmk}
Any Lie algebroid morphism $H\colon TI^2\to A$ can be written in the form
\begin{equation}\label{morphism}
H=a\,dt + b\, d\varepsilon
\end{equation}
where $a,b\colon I^2\to A$ are the $C^1$ functions:
\begin{equation*}
a(\varepsilon,t)=H\left({\partial_t}\right)(\varepsilon,t), \quad b(\varepsilon,t)=H\left(\partial_{ \varepsilon}\right)(\varepsilon,t).
\end{equation*}
In this presentation, the boundary conditions define the $A$-paths:
\begin{equation}\label{boundary conditions}
a(0,t)=a_0(t), \quad a(1,t)=a_1(t), \quad b(\varepsilon,0)=b(\varepsilon,1)=0.
\end{equation}
\end{rmk}
Thus, an $A-homotopy$ may be viewed as a pair of $C^1$ functions $a,b\colon I^2\to A$ satisfying the boundary conditions \eqref{boundary conditions} and combining, via the  
\eqref{morphism}
to a morphism of Lie algebroids. It is shown in \cite{zbMATH07480966} that $H=a\,dt + b\, d\varepsilon$ is a Lie algebroids morphism if and only if the pair $(a,b)$ solves the differential equation:
\begin{equation}\label{homotopyequation}
\partial_t a-\partial_\varepsilon b=\operatorname{T}_\nabla(a,b),
\end{equation}
where 
$T_{\nabla}$
is the torsion of an auxiliary $A$ connection on $A$, induced by a $TM$-connection on $A$ (a connection in the standard sense). The solutions of this equation do not depend on the choice of this connection.
\smallskip 

Let now $P(A)$ be the space of $A$-paths endowed with the $C^1$-topology. The quotient space
$$\mathcal{G}(A)=P(A)/\sim,$$
under the $A$-homotopy equivalence relation $\sim$, has a natural topological groupoid structure. It is called the \textit{Weinstein Groupoid}. If $\mathcal{G}$ integrates $A$, then this reconstructs the monodromy groupoid \eqref{MonG}.
Thus, for a Lie algebroid $A$, the Weinstein groupoid provides a natural candidate for a Lie groupoid integrating $A$. However, this groupoid carries only a topological groupoid structure. The obstruction to the existence of a smooth structure making $\mathcal{G}(A)$ into a Lie groupoid is encoded in the \textit{monodromy groups} associated with $A$.

\begin{rmk}
If $\g$ is a Lie algebra, seen as an algebroid over a point, $\mathcal{G}(\g)$ is the simply connected Lie group $\widetilde{G}$ with Lie algebra $\g$ provided by Lie's third theorem \cite{BAILEY_2001}. 
\end{rmk}

\subsubsection*{The monodromy map} Throughout this section we assume that $A$ is a transitive Lie algebroid over $M$. The construction for the general case follows  by restriction to the orbits. Let $m\in M$, $\mathfrak{g}=L_m$ be the isotropy Lie algebra at the point $m$ and let again $\widetilde{G}$ be the associated simply connected Lie group. The \textit{monodromy map} at the point $m$ is a group homomorphism 
\begin{equation*}
\partial_m\colon \pi_2(M,m)\to Z\tilde{G}
\end{equation*}
defined in the following way.
Let $[\varphi]\in \pi_2(M,m)$ be represented by a smooth map $\varphi\colon I^2 \to M$ collapsing the boundary to $m$. We can construct a Lie algebroid morphism $H=a \, dt+ b\, d\varepsilon\colon TI^2\to TM$ over $u $ with:
\begin{equation}\label{homotopyconditions}
a(0,t)=b(\varepsilon,0)=b(\varepsilon,1)=0.
\end{equation}
To find such a morphism, first choose any connection $\Theta\colon TM \to A$ and put $b(\varepsilon,t)=\Theta\left({\partial_{\varepsilon}\varphi}(\varepsilon,t)\right)$; then, take as $a\colon I^2\to A$, the solution of \eqref{homotopyequation} with initial condition $a(0,t)=0$.
By construction, the $A$-path $a_1(t)=a(1,t)$ stays in the isotropy Lie algebra $\mathfrak{g}$. Indeed, $\#(a_1(t))={\partial_t \varphi}(1,t)=0$. Therefore, $a_1$ can be integrated to a path $g(t)$ in $\widetilde G$ starting at the identity (a $\widetilde{G}$-path).
Finally define $$\partial_m([\varphi]):=g(1),$$ the endpoint of the integrated path. 
It can be proved that the definition does not depend on all the involved choices \cite{zbMATH02057953}. The image of the monodromy map is a  subgroup $\widetilde{\mathcal{N}}_m(A)$
of $\widetilde{G}$ called the \textit{monodromy group} of $A$ at the point $m$. 
It's proven that the monodromy groups are contained in the center $Z\widetilde{G}$ of $\widetilde{G}$.

\begin{thm}\cite{zbMATH02057953} Let $A$ be a transitive Lie algebroid over $M$. The following statements are equivalent:
\begin{itemize}
\item [(i)] $A$ is integrable.
\item [(ii)] $\mathcal{G}(A)$ is smooth.
\item [(iii)] The monodromy groups are discrete.
\end{itemize}
In this case, $\mathcal{G}(A)$ is the unique $s$-simply connected Lie groupoid integrating $A$.
\end{thm}
\begin{rmk}
Crainic and Fernandes identified obstructions to the integrability of arbitrary Lie algebroids. In the general case, the integrability criterion includes an additional condition which controls the discreteness of the monodromy groups in the direction transverse to the leaves of the singular foliation given by the orbits of the algebroid.
Indeed in the transitive case one proves that it is sufficient to have $\widetilde{N}_x(A)$ discrete for just one point (if one then all) see the Remark 4.12 in \cite{zbMATH02057953}. 
The relation between the various monodromy groups at different points in the same orbit is described in \cite{zbMATH07480966}.
\end{rmk}
We conclude the section by estabilishing the naturality of the monodromy with respect to pullbacks of Lie algebroids. The result follows from the fact that pullbacks preserve the isotropy Lie algebras, and that $A$-homotopies for the pullback induce $A$-homotopies for the original algebroid.

From now on, when needed, we will denote the monodromy map associated with the algebroid $A$ by $\partial_m^A$.
\begin{prop}Let $A$ be a transitive Lie algebroid over $M$, and  let $f\colon B\to M$ be a smooth map. Then
\begin{equation}\label{pullbackCrainic}
    \partial_m^{f^!A}=\partial_{f(m)}^A\circ f_*,
\end{equation}
where $f_*\colon \pi_2(B,m)\to\pi_2(M,f(m))$ is the induced morphism.
\end{prop}
\section{The Meinrenken's construction of the monodromy map}
In this section we describe an alternative way of constructing the monodromy map given by Meinrenken in \cite{zbMATH07480966}. Actually we present it in a slightly different way by following his lecture notes \cite{Meinnotes}. 
\subsubsection{Classification of transitive algebroids on the $2$--sphere.}
Using Theorem \ref{gauge1}, Meinrenken classified the transitive algebroids on the $2$-sphere with given framing. This is an extension of the clutching construction for principal bundles with given structure group $G$.

Fix a reference point on the standard sphere ${S}^2=\{x\in \mathbb{R}^3: \|x\|^2=1\}$; we shall always use the east pole $E=(0,1,0)$ for definiteness. 
 Let also
 $P \to {S}^2$ be a $G$-principal bundle. A framing for $P$ at $E$ is an isomorphism (of principal bundles) $i_{E}^*P \to G$. Of course a framing of principal bundles induces a framing for the corresponding transitive algebroid.
Let
$\operatorname{Prin}_G(S^2,E)$ be the set of isomorphism classes of $G$-principal bundles on the sphere $S^2$ with fixed framing at $E$. Pairs $(P,\varphi)$ are identified by isomorphisms which preserve the framings. We can make $\operatorname{Prin}_G(S^2,E)$ into an abelian group with operation induced by the connected sum ${S}^2 \, \sharp \,{S}^2 \cong {S}^2$. The usual clutching construction for principal bundles gives an isomorphism $\operatorname{Prin}_G(S^2,E)\cong \pi_1(G,1)$. 
\smallskip 

In the same way, if $\g$ is the Lie algebra of $G$,
we may consider the set
$\operatorname{Tran}_{\mathfrak{g}}(S^2,E)$
of isomorphism classes of 
pairs $(A,\varphi)$ where $A \Rightarrow S^2$ is a transitive algebroid and $\varphi:L_E \to \g$ is a fixed framing at $E$. Two pairs are identified if the corresponding algebroids are isomorphic and the isomorphism preserves the framings. In particular all the elements in $\operatorname{Tran}_{\mathfrak{g}}(S^2,E)$ represent transitive algebroids with the same isotropy type $\g$.
Also in this case there is a connected sum construction making it an abelian group.
The details can be found in the cited paper (and lecture notes). We won't use directly the group structure but notice that there is a natural group morphism $\operatorname{Prin}_G(S^2,E) \to \operatorname{Tran}_{\mathfrak{g}}(S^2,E)$
which maps a principal bundle to its associated algebroid.
\smallskip 

Let $\widetilde{G}$ be the simply connected Lie group integrating $\mathfrak{g}$; then $\widetilde{G}$ is the universal cover of the connected component of the identity $G^{\circ}$; the kernel of the covering map $\widetilde{G}\to G^{\circ}$ is identified with $\pi_1(G,1)=\pi_1(G^{\circ},1)$ and included in the center $Z\widetilde{G}$. 

Before stating the classification theorem
we fix an orientation of $S^1=S^2 \cap \{z=0\}$ in the counter-clockwise way and take the corresponding generator $1_{S}\in \pi_1(S^1,E)\cong \mathbb{Z}$.

 \begin{thm}{\emph{(}}Meinrenken \cite{zbMATH07480966,Meinnotes1}{\emph{)}} \label{thmMeinrenken}
 There is an isomorphism $\mathsf{c}:\operatorname{Tran}_{\mathfrak{g}}(S^2,E)\to Z\widetilde{G}$ 
 extending the clutching construction entering in a commutative diagram:
 \[\begin{tikzcd}
	{\operatorname{Prin}_G(S^2,E)} & {\operatorname{Tran}_{\mathfrak{g}}(S^2,E)} \\
	{\pi_1(G,E)} & {Z\widetilde{G}}
	\arrow[from=1-1, to=1-2]
	\arrow["\cong"', from=1-1, to=2-1]
	\arrow["\mathsf{c}", from=1-2, to=2-2]
	\arrow[from=2-1, to=2-2]
\end{tikzcd}\]
\end{thm}
\begin{proof}
We give a sketch of the proof which serves to define the map $\mathsf{c}$. Let $V_{\pm}$ be the contractible open sets obtained by removing the South/North poles. Then the annulus $V:=V_+ \cap V_-$ retracts to $S^1$. An algebroid $A$ with fixed framing at $E$ determines  (not unique)  trivializations $\Psi_{\pm}:TV_{\pm} \times \mathfrak{g} \to A\rest{V_{\pm}} $ both preserving the framing.
The gluing function $\Psi_-^{-1}\circ\, (\Psi_{+}\rest{TV \times \g})$
is an automorphism of the trivial algebroid $TV\times \mathfrak{g}$ fixing the framing.
By Theorem \ref{gauge1} this corresponds to a pair $(\rho,f)$ with $\rho:\pi_1(V,E)\to Z\widetilde{G}$ and $f: \widetilde{V}\to \widetilde{G}$ quasi periodic with respect to $\rho$. 

Finally, with a small abuse of notation (about the class of $A$), the isomorphism $\mathsf{c}$ is given by $$\mathsf{c}(A):=\rho(1_S).$$
It can be proven that it is well defined.
 Moreover in the case of a principal bundle, $\mathsf{c}$ is realized as the parallel transport with respect to some connection and the diagram commutes.
\end{proof}

\subsubsection*{The monodromy map}
Let $\xymatrix{L \ar@{^{(}->}^{}[r]&A\ar@{>>}[r]^-{\sharp}&TM}$; be a transitive algebroid. 
Fixing a point $m\in M$ serves to identify a 
 fiber $\mathfrak{g}=L_{m}$. 
 The monodromy map at $m$ as described by Meinrenken is defined by: $$c_m: \pi_2(M,m) \longrightarrow Z\widetilde{G}, \quad [u] \longmapsto \mathsf{c}(u^{!}A).$$
  Here, we are representing $\pi_2$ cycles with smooth maps $u: S^2 \to M$ such that $u(E)=m$.
  
\section{Identification of the Monodromy Homomorphisms}
Previously, we introduced the monodromy homomorphisms defined by Crainic--Fernandes and the version due to Meinrenken. In this section we prove that the two constructions coincide. We first establish the result for Lie algebroids over the sphere and then extend it to arbitrary manifolds using the functoriality with respect to pullbacks of algebroids.

We begin with the following observation.
\begin{lemma}
\label{homotopypaths}
Let A be a Lie algebroid over a manifold $M$ and let $H\colon TI^2\to A$ be an $A$-homotopy. Then,
\begin{itemize}
\item the A-paths $a_0=\left.H\right|_{T\left( \{0\}\times I \right)}$ and $a_\Delta=\left.H\right|_{T\left( \Delta I^2 \right)}$ are A-homotopic.
\item The A-paths $a_\Delta$ and $a_1=\left.H\right|_{T\left( \{1\}\times I \right)}$ are A-homotopic.

\end{itemize}
In both cases, the A-homotopy is induced by the homotopy between the vertical line $\{0\}\times I$ {\em{(}}resp. $\{1\}\times I${\em{)}} and the diagonal $\Delta I^2=\{(t,t)|t\in I\}$. Moreover, the homotopy is contained in the upper {\em{(}}resp. lower {\em{)}} triangular region with respect to the diagonal.
\end{lemma}
\begin{thm}\label{Crainic=Meinrenken}
For any transitive Lie algebroid on a manifold $M$, the Crainic--Fernandes and the Meinrenken monodromy maps agree.
\end{thm}
\begin{rmk}\label{phi}
Before giving the proof, we note that two equivalent models of $\pi_2(M,m)$ are being used. In the construction of the Crainic-Fernandes monodromy, $\pi_2(M,m)$ uses smooth functions from the square collapsing the boundary to the reference point, wherease in the Meinrenken's construction, $\pi_2(M,m)$ is described using smooth maps from the sphere. These two models can be canonically identified by fixing a smooth map $\varphi\colon I^2\to S^2$, that is constant on the boundary and restricts to an orientation preserving diffeomorphism on $I^2-\partial I^2$.  The identification given by the composition with $\varphi$ is an isomorphism unique up to isotopy.
\end{rmk}

\begin{proof}
We first prove the theorem for $M=S^2$ with base point the east pole $E$ and then extend it to the general case by naturality. 
Let $\g=L_E$ and choose the framing: $\operatorname{Id}:\g \to L_E.$

Let $V_{\pm}$ be
the contractible open sets obtained by removing the south and north poles, respectively. Put $V:=V_+ \cap V_-$. As in the proof of \eqref{thmMeinrenken}, fix trivializations $\Psi_{\pm}: TV_{\pm} \times \mathfrak{g} \to A\rest{V_{\pm}}$.

Now the requirement that these trivializations preserve the framing means that $\Psi_{\pm}$ are the identity when restricted to the fiber of $L$ over $E$.

Let $\widetilde{V}$ be the universal cover of $V$; we realise it as the space of homotopy classes of paths starting from $E$
(the $s$-fiber of the fundamental groupoid).
In particular 
we fix the point $\widetilde{E}$ above $E$ corresponding to the trivial path at $E$.
With these choices, 
the automorphism  $\Psi_{-|+}:=\Psi_-^{-1}\circ\, (\Psi_{+}\rest{TV \times \g})\in \operatorname{Aut}(TV \times \mathfrak{g})$ corresponds to a pair $(\rho,f)$ with $f(\widetilde{E})=1$ and, by definition
$$\mathsf{c}(A)=\rho(1_S).$$
Here $1_S$ denotes the generator of $\pi_1(V,E)\cong \mathbb{Z}$ represented by the smooth loop $\gamma\colon I\to S^2$ based at $E$ and parametrizing the equator $S^1 \subset S^2$
with the choices already made before Theorem \ref{thmMeinrenken}.

Let $\varphi\colon I^2\to S^2$ be a smooth map as in Remark \ref{phi} representing the fundamental class of $\pi_2(S^2,E)\cong \mathbb{Z}$. It may be chosen so that $\left.\varphi\right|_{\Delta I^2}=\gamma$, and such that the image of the open upper (respectively, lower) triangular region of $I^{2}$, with respect to the diagonal $\Delta I^{2}$, contains the north (respectively, south) pole. Let $H=a\,dt+b\,d\varepsilon \colon TI^2\to A$ be an A-homotopy over $\varphi$, satisfying \eqref{homotopyconditions}, and let $a_0,a_1,a_\Delta$ be the $A$-paths as in Lemma \ref{homotopypaths}. Now define the $TV_\pm \times \mathfrak{g}\,$-paths $$\alpha_{\pm}:=\Psi_{\pm}^{-1}\circ a_{\Delta}=\dot\gamma+\xi_{\pm},$$
where $\xi_\pm$ are $\mathfrak{g}\,$-paths. These are also
$TV\times\mathfrak{g}\,$-paths and, by construction,
\begin{equation}\label{transition}
\Psi_{-|+}(\alpha_+(t))=\alpha_-(t).
\end{equation}

Moreover, by the Lemma \ref{homotopypaths} and the boundary conditions \eqref{homotopyconditions}:
\begin{itemize}
\item $\alpha_+$ is $TV_+\times \mathfrak{g}\,$-homotopic to $\Psi_+^{-1}\circ a_0$, hence to the zero $TV_+\times \mathfrak{g}\,$-path,
\item $\alpha_-$ is $TV_-\times \mathfrak{g}\,$-homotopic to $\Psi_-^{-1}\circ a_1$. Since $\Psi_-$ preserves the framing at $E$, we have $\Psi_-^{-1}\circ a_1(t)=0\oplus a_1(t)$.
\end{itemize}

The $TV\times\mathfrak{g}\,$-paths $\alpha_\pm$ integrate to $\mathcal{G}$-paths $\eta_{\pm}(t)=([\widetilde\gamma(t),\widetilde{E}],g_{\pm}(t))$, where $\mathcal{G}:=\dfrac{\widetilde{V}\times \widetilde{V}}{\pi_1(V,E)}\times \widetilde{G}$. Here $\widetilde{\gamma}$ denotes the lift of $\gamma$ starting at $\widetilde{E}$, and $g_\pm$ are the $\widetilde{G}$-paths integrating $\xi_{\pm}$. 
The automorphism $\Psi_{-|+}$ integrates to the automorphism $F$ of $\mathcal{G}$, given by
$$ F([\widetilde{y},\widetilde{x}],g)=([\widetilde{y},\widetilde{x}],f(\widetilde{y})gf(\widetilde{x})^{-1}).$$
 Integrating \eqref{transition}, we have
$$ F(\eta_+(t))=\eta_-(t).$$
Comparing the second components, it follows that 
$$ f(\widetilde\gamma(t))g_+(t)f(\widetilde{E})^{-1}=g_-(t). $$
Since $\alpha_+$ is $TV_+\times \mathfrak{g}\,$-homotopic to the zero $TV_+\times \mathfrak{g}\,$-path, we have $g_+(1)=1$. Since $\alpha_-$  is $TV_-\times\mathfrak{g}\,$-homotopic to $\Psi_-^{-1}\circ a_1$, it follows that $g_-$ is homotopic to the integral path of $a_1$ inside $\widetilde{G}$, then $$g_-(1)=\partial([\varphi])=\partial(1_{S^2}).$$ 
Here we have used the fact that $\alpha_{+}$ also integrates to a unique
$(V_+\times V_+ \times{\widetilde{G}})$-path whose third component is necessarily given by $g_{+}(t)$.
In other words, the inclusion $TV \times \mathfrak{g} \to TV_+ \times \mathfrak{g}$ integrates to a groupoid morphism $\dfrac{\widetilde{V}\times \widetilde{V}}{\pi_1(V,E)}\times \widetilde{G}\to V_+ \times V_+ \times \widetilde{G}$ explicitly given by $([\tilde{x},\tilde{y}],g)\mapsto (\tilde{x},\tilde{y},g).$
The same argument applies to $\alpha_{-}$.
Finally using $f(\widetilde{E})=1$ and $f(\widetilde\gamma(1))=f(\widetilde{E}\cdot [\gamma])=\rho([\gamma])$, we conclude
$$ \partial(1_{S^2})=\rho(1_S)=c(A). $$

Now let $A$ be a Lie algebroid over a general manifold $M$. By naturality, for any smooth map $u\colon S^2 \to M$, 
$$ c_E([u])=\mathsf{c}(u^!(A))=\partial_{u^!(A)}(1_{S^2})=(\partial_{A}\circ u_*)(1_{S^2})=\partial_A ([u_0]),$$
where $u_0=\sigma\circ u \colon I^2 \to M$.
\end{proof}

  \subsection{The monodromy map as the Mackenzie class}
  There remains to identify the monodromy map $c: \pi_2(M,m) \to Z\widetilde{G} $ defined by Meinrenken with the Mackenzie obstruction class. 
  Recall the natural isomorphism $\mathscr{C}:\check{H}^*(M;G)\to H^*_{\operatorname{sing}}(M;G)$; we combine it with the natural pairing $$H^2_{\operatorname{sing}}(M;Z\widetilde{G}) \times \pi_2(M,m)\to Z\widetilde{G}, \quad \langle \omega,[u]\rangle =\langle \omega, u_*[1_{{S}^2}]\rangle , $$
  where
  $u_*:H_2(S^2) \to H_2(M)$ is the induced map and
  $1_{S^2}\in H_2(S^2)$ is the homology fundamental class (integer coefficients).

\begin{thm}\label{mcmein}                               
Let $A$ be a transitive algebroid 
on a simply connected manifold $M$
with reference point $m\in M$. Let $\widetilde{G}$ be the simply connected Lie group integrating $\mathfrak{g}=L_{m}$. For the Mackenzie class and the monodromy map we have: 
 $$\langle \mathscr{C}(e^A), \cdot \, \rangle =-\partial_m $$
as morphisms $\pi_2(M,m) \to Z\widetilde{G}.$

 \end{thm}
 \begin{proof}
By Theorem \ref{Crainic=Meinrenken} it sufficies to prove the statement using the monodromy map $c_m$.

The proof consists of two steps:
\begin{enumerate}
\item the case of $S^2$;
\item the general case, which follows by naturality of the constructions.
\end{enumerate}

\noindent (1) 
Let $A$ be a transitive algebroid on $S^2$ with reference point $m=E$.
We make all the choices we already made in the proof of proposition \ref{fundamentalpairing}; in adjoint we require that $\mathsf{e}_i$ correspond to the standard points $\mathbb{R}^3$ and the baricenter of the base is mapped to the south pole. In particular $\mathsf{e}_2=E$.

Also consider the open cover $\mathcal{V}=\{V_+,V_-\}$ of $S^2$ with 
$$V_+=U_0 \cup U_1 \cup U_2, \quad V_-=U_3.$$

 Then of course $\mathcal{U}$ is a refinement of $\mathcal{V}$ with refining function $\lambda:\{0,1,2,3\}\to \{+,-\}$ such that
$\lambda(i)=+$ for $i=0,1,2$ and $\lambda(3)=-$.


The idea is to compute the Mackenzie class with respect to $\mathcal{U}$ using an algebroid atlas compatible with Meinrenken's description of the gauge transformations. 

So we start with trivializations $\overline{\psi}_{\pm}:V_{\pm} \times \g \to L\rest{V_{\pm}}$ corresponding to sections of $r:\operatorname{Int}(A)_E \to S^2$ over $V_{\pm}$, chosen in a way that these map $E$ to the identity. The induced framings at $E$ are both the identity of $\g=L_E$, and 
$V=V_+ \cap V_-$ is an open annulus retracting to to the equator $S^1$. 
The gluing function $\overline{a}_{-|+}$, defined over $V$ by $\overline{\psi}_-^{\,-1} \circ \left( \overline{\psi}_+\,\rest{V\times \g}\right)$ takes its values in $\operatorname{Int}(A)_E^E=\operatorname{Ad}(\widetilde{G})$. 
Theorem \ref{completiontheorem}, ensures the existence of flat connections $\overline{\Theta}_{\pm}$ and an algebroid atlas $\overline{\Psi}_{\pm}:TV_{\pm} \times \g \to A\rest{V_{\pm}}$ in the form $\overline{\Psi}_{\pm}=\overline{\Theta}_{\pm} + \overline{\psi}_{\pm}$.

We have a transition automorphism
$$\Psi_{-|+}:=\left(\overline{\Psi}_-\right)^{\,-1} \circ \left( \overline{\Psi}_+\,\rest{TV\times \g}\right): TV\times \g \longrightarrow TV \times \g,$$
$$X+\xi \longmapsto X+\overline{\chi}_{-|+}(X)+ \overline{a}_{-|+}(\xi),$$ with $$\overline{\chi}_{-|+}:=\overline{\psi}_-^{\,-1}\,\circ \overline{\ell}_{-|+}, \quad \overline{\ell}_{-|+}=\overline{\Theta}_+ -\overline{\Theta}_-.$$
 The framing at $E$ is preserved, so by Theorem \ref{gauge1} this is completely described by a representation
$\rho:\pi_1(V,E) \longrightarrow Z\widetilde{G}$ and a map $f$
defined on the universal cover of $V$ denoted by 
$ \widetilde{V}$.
The map $f: \widetilde{V}\longrightarrow \widetilde{G}$ is such that $f(\widetilde{E})=1_{Z\widetilde{G}}$ and $f(x \cdot \gamma)=f(x)\rho(\gamma)$ for every $x\in \widetilde{V}$ and $\gamma \in \pi_1(V,E)$.
Here $\widetilde{E}$ could be any fixed point in $\widetilde{V}$ above $E$. In particular we may keep the same choices as in the proof of Theorem \ref{Crainic=Meinrenken}.

 From \eqref{Conditionspullbackgauge}, it follows that $\Delta(f)=\pi^* (\overline{\chi}_{-|+})$ and $\operatorname{Ad}_{f(x)}= \overline{a}_{-|+}(\pi(x))$, for every $x\in \widetilde{V}$. 
Moreover, by definition $$\mathsf{c}(A)=\rho(1_S).$$

Now we come back to the good cover $\mathcal{U}$. 
The above choices induce Mackenzie's datas by:
 $$ \psi_i:=\overline{\psi}_{\lambda(i)}\rest{U_i \times \g} \quad \textrm{and} \quad 
\Theta_i:=\overline{\Theta}_{\lambda(i)}\rest{U_i}.$$

Then $\Psi_i:=\Theta_i + \psi_i$ define an algebroid atlas over $\mathcal{U}$ whose 
associated system of transition datas $(U_i,a_{ij},\chi_{ij})$ has the property:								
$$\chi_{ij}=\begin{cases}
\overline{\chi}_{+|-}\,\rest{TU_{i3}} & \text{if } i=0,1,2,\, j=3\\
0 \,\rest{TU_{ij}}& \text{otherwise}
\end{cases}, \quad a_{ij}=\begin{cases}
\overline{a}_{+|-}\,\rest{U_{i3}\times\mathfrak{g}} & \text{if } i=0,1,2,\, j=3\\
\operatorname{id} \,\rest{U_{ij}\times\mathfrak{g}}& \text{otherwise}
\end{cases}.$$
This system meets all the requirements to compute Mackenzie class.

Now, for $i=0,1,2$ we have $U_{i3}=U_i \cap U_3 \subseteq V$; any choice of sections 
$\eta_{i3}: U_{i3} \longrightarrow \widetilde{V}$ will satisfy:
$$\Delta(f^{-1} \circ \eta_{i3})=\overline{\chi}_{+|-}\rest{U_{i3}}=\chi_{i3}, \quad i=0,..,2.$$

In particular to represent $e^A$ we may take $s_{i3}:=f^{-1}\circ \eta_{i3}$ for $i=0,1,2$ and $j=3$ and the identity of $\widetilde{G}$ otherwise.

We make a further choice for the sections $\eta_{i3}$ i.e. we require: 

$$\eta_{13}(e_2)=\widetilde{E}, \quad \eta_{1,3}(\mathsf{e}_0)=\eta_{2,3}(\mathsf{e}_0) \quad  \textrm{and}\quad \eta_{2,3}(\mathsf{e}_1)=\eta_{0,3}(\mathsf{e}_1).$$

This is possible and correponds to the lifting of  $1_S \in \pi_1(V,E)$
with starting point $\widetilde{E}$.
Therefore for the resulting endpoint we have:
$$\eta_{0,3}(E)=\eta_{1,3}(E) \cdot 1_S, \quad 1_S \in \pi_1(V,E).$$

Finally we compute the value of the Mackenzie class $e$ constructed with these choices on the fundamental cycle of $S^2$.
By proposition \ref{fundamentalpairing} we have:
\begin{align*}
\langle \mathscr{C}[e^A], 1_{S^2} \rangle &=e_{(0,1,2)}(\mathsf{e}_3)\cdot e_{(0,1,3)}(\mathsf{e}_2)^{-1}\cdot e_{(0,2,3)}(\mathsf{e}_1)\cdot e_{(1,2,3)}(\mathsf{e}_0)^{-1}
\\ & = s_{12}s_{02}^{-1}s_{01}(\mathsf{e}_3)\cdot s_{13}s_{03}^{-1}s_{01}(\mathsf{e}_2)^{-1}\cdot s_{23}s_{03}^{-1}s_{02}(\mathsf{e}_1)\cdot s_{23}s_{13}^{-1}s_{12}(\mathsf{e}_0)^{-1}
\\ & = s_{03}s_{13}^{-1}(\mathsf{e}_2)\cdot s_{23}s_{03}^{-1}(\mathsf{e}_1)\cdot s_{13}s_{23}^{-1}(\mathsf{e}_0)
\\ & = f(\eta_{03}(\mathsf{e}_2))^{-1} f(\eta_{13}(\mathsf{e}_2))+ f(\eta_{23}(\mathsf{e}_1))^{-1} f(\eta_{03}(\mathsf{e}_1))+f(\eta_{13}(\mathsf{e}_0))^{-1}f(\eta_{23}(\mathsf{e}_0))\\ & = f(\widetilde{E} \cdot 1_S)^{-1} =\rho(1_S)^{-1}=-\rho(1_S)=-\mathsf{c}(A).
\end{align*}

\noindent (2) If $A$ is a transitive Lie algebroid over a simply connected manifold $M$, the result follows by applying the previous step and Proposition \ref{functorialityMackenzie} to the pullback $u^!A$,  where $u\colon S^2\to M$ is a smooth map sending $E$ to a reference point $m$. This concludes the proof.
\end{proof}
\begin{corl}Let $A$ be a transitive algebroid on a manifold $M$ with reference point $m\in M$. Let $\widetilde{M}$ be the universal cover with projection $\pi\colon \widetilde{M}\to M$ and reference point $\widetilde{m}$ over $m$. By definition $e^A:=e^{\pi^!A}\in \check{H}^2(\widetilde{M};ZG) $. Then,
$$\langle \mathscr{C}_{\widetilde{M}}(e^A), \, \cdot \, \rangle =-\partial_{\widetilde{m}}\circ \pi_*,$$
\end{corl}
where $\pi_*\colon \pi_2(\widetilde{M},\widetilde{m})\to \pi_2(M,m)$ is the induced isomorphism.
 \appendix

 \section{\v{C}ech to singular cohomology}

In this section we compare \v{C}ech and singular cohomology 
with coefficients in an abelian group $G$.
The topic is classical and can be found in many textbooks. 
Since the isomorphism is crucial in our treatment we include a somewhat detailed discussion
using a bicomplex as in Weil's proof of the de Rham Theorem \cite{zbMATH03073643} (cf. \cite{zbMATH07420243,zbMATH03581280,zbMATH03782042}).

We shall deal with smooth manifolds even if the results in this section hold in more generality (paracompact locally contractible spaces) perhaps with minor modifications in the proofs.

\subsubsection*{ \v{C}ech cohomology}
Let $\mathcal{F}$ be a presheaf on a manifold $M$ 
and fix any open cover $\mathcal{U}=\{U_{\alpha}\}_{\alpha \in \mathcal{A}}$ of $M$. It is sufficient to restrict to locally finite countable covers.

For a subset $A\subset \mathcal{A}$ put: $U_A:=\bigcap_{a\in A}{U_a}$ and denote with $N(\mathcal{U})$ the nerve of the covering. This is the collection of all the finite subsets $A$ of $\mathcal{A}$ such that $U_A \neq \emptyset$. The nerve is an (abstract) simplicial complex:
any element $A \in N(\mathcal{U})$ is 
 a face. If $|A|=q+1$, we say that $A$ is $q$-dimensional. An (ordered) $q$-simplex of $N(\mathcal{U})$ is a map $\sigma: \{0,1,...,q\} \to 
 \mathcal{A}$ such that $\sigma(\{0,1,...,q\}) \in N(\mathcal{U})$. Given a $q$-simplex $\sigma$, we set
  $U_{\sigma}:=\bigcap_{a\in \sigma(\{0,1,...,q\})} U_{a}$ and
denote with $S_q(\mathcal{U})$ the set of all the $q$-dimensional  simplexes.
We have an associated 
\v{C}ech complex
\begin{equation}\label{cechcomplex}
\begin{tikzcd}
	{\check{C}^0(\mathcal{U};\mathcal{F}}) & {\check{C}^1(\mathcal{U};\mathcal{F})} & {\check{C}^2(\mathcal{U};\mathcal{F})} & \cdots
	\arrow["\delta", from=1-1, to=1-2]
	\arrow["\delta", from=1-2, to=1-3]
	\arrow["\delta", from=1-3, to=1-4]
\end{tikzcd}
\end{equation}
with $\check{C}^q(\mathcal{U};\mathcal{F})=\prod \limits_{\sigma \in S^q(\mathcal{U})}\mathcal{F}(U_{\sigma})$ and the
 $\delta$'s are the usual differentials
 coming from the simplicial structure.
A $q$-cochain $\varphi \in \check{C}^q(\mathcal{U};\mathcal{F})$ is therefore a collection $\varphi=(\varphi_{i_0 \cdots i_q})_{i_0\cdots i_q} $ where
$\varphi_{i_0 \cdots i_q}\in \mathcal{F}(U_{i_0 \cdots i_q})$
for every $(i_0,...,i_q) \in \mathcal{A}^q$. We are of course abbreviating the notation: $U_{i_0\cdots i_q}:=U_{\{i_0,...,i_q\}}.$

 The cohomology of the complex \eqref{cechcomplex} is denoted with $\check{H}^*(\mathcal{U};\mathcal{F})$.
This construction is well behaved with respect to refinements
(different refining maps yield chain homotopic maps).
The \v{C}ech cohomology groups of $M$ with coefficients in $\mathcal{F}$, denoted by ${\check{H}}^*(M;\mathcal{F})$
are defined as the colimit over all the open covers.
They are functorial with respect to continuous maps and morphisms of presheaves.

When $G$ is an abelian group - the case of the Mackenzie class - the \v{C}ech cohomology groups with coefficients in the constant presheaf $U \mapsto G$ are simply denoted with ${\check{H}}^*(M;G)$. 

On paracompact spaces the \v{C}ech cohomology groups of a presheaf coincide with the ones of the associated sheaf \cite[Proposition 13.16]{zbMATH07380471} so that we could use the sheaf of locally constant functions with values in $G$.

\subsubsection{Ordered vs unordered coverings} \label{ordering}
Instead of considering all the simplexes of the nerve we may use only the oriented ones. First choose a total order on $\mathcal{A}$ and denote with $\underrightarrow{S}_q(\mathcal{U})$ the set of oriented $q$-simplexes, the strictly increasing maps $\{0,1,...,q\} \to \mathcal{A}$.
This yields a smaller complex $\underrightarrow{\check{C}}^*(\mathcal{U};\mathcal{F})$, we call it 
{\em{the ordered complex}},
chain homotopic to the {\em{unoredered complex}} ${\check{C}}^*(\mathcal{U};\mathcal{F})$.
The chain homotopy is realized by the natural projection ${\check{C}}^*(\mathcal{U};\mathcal{F}) \to \underrightarrow{\check{C}}^*(\mathcal{U};\mathcal{F})$. This can be easily proven with the methods of Serre's seminal paper \cite[Section 20]{zbMATH03782042}. Notice that Serre uses the complex of alternating chains, a subcomplex of ${\check{C}}^*(\mathcal{U};\mathcal{F})$
isomorphic to the ordered one. 

While the unordered complex is more natural, suited for general proofs,
 the ordered one is often more manageable in computations.
 
\subsubsection*{Comparison with singular cohomology}
Let $M$ be a manifold and fix an abelian group $G$. It is well known that there is a natural isomorphism 
$$H^*_{\textrm{sing}}(M;G) \cong \check{H}^*(M;{G})$$
between the singular cohomology with coefficients in $G$ and the \v{C}ech cohomology with coefficients in ${G}$. This can be done in different ways \cite{zbMATH07420243}; we use a \v{C}ech to de Rham type bicomplex.



For an open set $U \subset M$, we denote by $C^p_{\textrm{sing}}(U;G)$, $p\in \mathbb{N}$ the group of $p$-singular cochains over $U$. 
Recall that these are maps associating an element of $G$ to any singular $p$-simplex $\Delta_p \to U$. Letting $U$ vary, we have a presheaf $U \mapsto C^p_{\textrm{sing}}(U;G).$
We denote with $d:C^p_{\textrm{sing}}(U;G) \to C^{p+1}_{\textrm{sing}}(U;G)$ the usual coboundary induced by the coboundary operator of the singular chains.

Let now $\mathcal{U}=\{U_{\alpha}\}_{\alpha \in \mathcal{A}}$ be a good cover\footnote{the cover is locally finite and all the non void intersections of the $U_i$'s are contractible}
of $M$ and consider the double complex $E^{p,q}:=\prod_{\sigma \in S_p(\mathcal{U})} C^q_{\textrm{sing}}(U_{\sigma},G),$ $p,q\in \mathbb{N}$ equipped with horizontal differential $\delta$ and vertical differential $d$ respectively induced by the \v{C}ech and singular coboundary.

In particular the $q$-th row of the bicomplex is the \v{C}ech complex associated to $\mathcal{U}$ for the presheaf $U \mapsto C^{q}_{\textrm{sing}}(U;G)$ while the $p$-th column is the product over all the $p+1$ intersections of the singular complexes.

Since the two differentials commute, as usual we make a single complex $(E_i,D)$ out of it with $$E_i=\bigoplus \limits_{p+q=i}E^{p,q}, \quad  D=\delta+(-1)^pd,$$  with total cohomology ${H}_D^*(E)$.

The bicomplex can be augmented by adding the \v{C}ech and the singular complex:
\[\begin{tikzcd}[sep=small]
	{C^2_{\textrm{sing}}(M;{G})} & {\prod \limits_{\sigma \in S_0(\mathcal{U})}C^2_{\textrm{sing}}(U_{\sigma};{G})} & {\quad \quad } \\
	{C^1_{\textrm{sing}}(M;{G})} & {\prod \limits_{{\sigma \in S_{0}(\mathcal{U})}}C^1_{\textrm{sing}}(U_{\sigma};{G})} & {\prod \limits_{\sigma \in S_1(\mathcal{U})}C^1_{\textrm{sing}}(U_\sigma;{G})} & {\quad \quad } \\
	{C^0_{\textrm{sing}}(M;{G})} & {\prod \limits_{\sigma \in S_0(\mathcal{U})}C^0_{\textrm{sing}}(U_{\sigma};{G})} & {\prod \limits_{\sigma \in S_1(\mathcal{U})}C^0_{\textrm{sing}}(U_{\sigma};{G})} & {\prod \limits_{\sigma \in S_2(\mathcal{U})}C^0_{\textrm{sing}}(U_{\sigma};{G})} & {} \\
	& {\check{C}^0(\mathcal{U};{G})} & {\check{C}^1(\mathcal{U};{G})} & {\check{C}^2(\mathcal{U};{G})} & {}
	\arrow["\mathsf{r}", from=1-1, to=1-2]
	\arrow[shorten <=20pt, from=1-2, to=1-3]
	\arrow["d", from=2-1, to=1-1]
	\arrow["\mathsf{r}", from=2-1, to=2-2]
	\arrow["d", from=2-2, to=1-2]
	\arrow[from=2-2, to=2-3]
	\arrow[shift left, shorten <=14pt, shorten >=11pt, from=2-3, to=2-4]
	\arrow["d", from=3-1, to=2-1]
	\arrow["\mathsf{r}", from=3-1, to=3-2]
	\arrow["d", from=3-2, to=2-2]
	\arrow[shift left, from=3-2, to=3-3]
	\arrow[from=3-3, to=2-3]
	\arrow[shift left, from=3-3, to=3-4]
	\arrow[shorten >=1pt, from=3-4, to=3-5]
	\arrow["h", from=4-2, to=3-2]
	\arrow["\delta", from=4-2, to=4-3]
	\arrow["h", from=4-3, to=3-3]
	\arrow["\delta", from=4-3, to=4-4]
	\arrow["h", from=4-4, to=3-4]
	\arrow[shorten >=9pt, from=4-4, to=4-5]
\end{tikzcd}\]

\begin{thm}\label{cohomoth}
The rows $E^{*,q}$ of the (non augmented) bicomplex are exact in degree $p>0$ and the columns of the augmented complex are exact.
The maps $\mathsf{r}$ and $h$, induced by restrictions, give rise to isomorphisms
 $$h_*: \check{H}^*(\mathcal{U};G) \to {H}_D^*(E) \quad \textrm{   and   } 
\quad \mathsf{r}_*:H^*_{\textrm{sing}}(M;G) \to {H}_D^*(E).$$ 
In terms of spectral sequences, both the spectral sequences induced by the filtrations by rows and columns degenerate at the second page.
In particular we get an isomorphism 
$\check{H}^*(\mathcal{U};G)\cong H^*_{\operatorname{sing}}(M;G)$ and in turn an isomorphism
$\mathscr{C}:\check{H}^*(M;G)\to H^*_{\operatorname{sing}}(M;G)$ which is natural. Any continuous map $f:M \to N$ induces a commutative diagram:
\[\begin{tikzcd}
	{\check{H}^*(N;G)} & {H^*_{\textrm{sing}}(N;G)} \\
	{\check{H}^*(M;G)} & {H^*_{\textrm{sing}}(M;G)}
	\arrow["\mathscr{C}", from=1-1, to=1-2]
	\arrow["{f^*}"', from=1-1, to=2-1]
	\arrow["{f^*}", from=1-2, to=2-2]
	\arrow["\mathscr{C}", from=2-1, to=2-2]
\end{tikzcd}\]			
\end{thm}
\begin{proof}
Start with the Mayer--Vietoris sequence for singular chains
$$\xymatrix{0& C_{q,\, \mathcal{U}}^{\textrm{sing}}(M)\ar[l]_-{}&\bigoplus \limits_{\sigma \in S_0(\mathcal{U})}C_q^{\textrm{sing}}(U_{\sigma})\ar[l]_-{\varepsilon}&\bigoplus \limits_{\sigma \in S_1(\mathcal{U})}C_q^{\textrm{sing}}(U_{\sigma})\ar[l] & \cdots \ar[l],
}$$ Here for an open $U$ the term $C_q^{\textrm{sing}}(U)$ is the free group generated by the singular simplexes contained in $U$. The group $C_{q,\, \mathcal{U}}^{\textrm{sing}}(M)\subseteq C_{q}^{\textrm{sing}}(M)$ is the free group of the $\mathcal{U}$-small simplexes, i.e. the subgroup generated by the singular simplexes contained in some open set of $\mathcal{U}$. The exactness of this sequence is proven in \cite[Proposition 15.22]{zbMATH03782042}. Every term is free so that dualization produces the exact sequence
$$\xymatrix{0\ar[r]& \operatorname{Hom}_{\mathbb{Z}}(C_{q,\, \mathcal{U}}^{\textrm{sing}}(M);G)\ar[r]^-{\varepsilon^*}&{\prod \limits_{{\sigma \in S_{0}(\mathcal{U})}}C^q_{\textrm{sing}}(U_{\sigma};{G})}\ar[r]^-{\delta}&\prod \limits_{{\sigma \in S_{1}(\mathcal{U})}}C^q_{\textrm{sing}}(U_{\sigma};{G})\ar[r]& \cdots 
}$$
Such dual sequence, from the second position on, is exactly the $q$-th row of the complex.
Therefore $${\phantom{x}'}E^{p,q}_1
\cong H^p(E^{*,q},\delta)=\begin{cases}
			\varepsilon^*\operatorname{Hom}_{\mathbb{Z}}(C_{q,\, \mathcal{U}}^{\textrm{sing}}(M);G), & \text{if $p=0$}\\
            0, & \text{otherwise.}
		 \end{cases}$$
Only the first column survives and yields the complex
$$({\,}'E_1^{\,0,\ast},d)=\left( \varepsilon^*\operatorname{Hom}_{\mathbb{Z}}(C_{q,\, \mathcal{U}}^{\textrm{sing}}(M);G), d\right) $$
where the differential is induced by $d$. Let's prove that the map $\mathsf{r}$ takes its values in 
${\,}'E_1^{\,0,\ast}$ and
 induces isomorphism in cohomology. This is immediate because the inclusion $\iota_{\mathcal{U}}:C_{q,\, \mathcal{U}}^{\textrm{sing}}(M)\hookrightarrow C_{q}^{\textrm{sing}}(M)$ is a chain homotopy equivalence
by the small simplexes method \cite[Appendix A]{zbMATH03413541}. Therefore $\iota_{\mathcal{U}}^*$ is a chain homotopy equivalence, $\varepsilon^*$ is an isomorphism into its image
and we have $\mathsf{r}=(\iota_{\mathcal{U}}\circ \varepsilon)^*$.

Taking the $d$-cohomology:
$${\phantom{.}'}E_2^{p,q}\cong 
H^q\left(\phantom{.}'E_1^{p,\ast},d \right)
=\begin{cases}
			H^q_{\textrm{sing}}(M;G), & \text{if $p=0$}\\
            0, & \text{otherwise}
		 \end{cases}$$
and we have a natural isomorphism $H_D(E^{\ast,\ast})\cong H^*_{\textrm{sing}}(M;G)$.
The isomorphism takes a representative in the total cohomology and reduces it to a representative in ${\phantom{.}'}E_2^{0,q}$.
In particular the spectral sequence obtained by filtering the double complex by rows
degenerates at the second page.

Now we look at the augmented columns; these are exact in strictly positive degree because $\mathcal{U}$ is a good covering and compute the singular cohomology of the open sets and their intersections. In degree $0$ at the $p$-th column, we have 
$\operatorname{ker}d=\prod \limits_{\sigma \in S_p(\mathcal{U})}C_{\textrm{const}}(U_{\sigma};{G})$, the direct product of the groups of constant functions with values in $G$. This is exactly the augmented row. It follows that 
only the first row survives after taking the $d$-cohomology and 
$h$ induces isomorphism with 
$$E_2^{p,q}\cong 
\begin{cases}
			\check{H}^p(\mathcal{U};{G}), & \text{if $q=0$}\\
            0, & \text{otherwise}
		 \end{cases}$$
 which is, in turn, isomorphic to $H_D^*(E)$.
 Also the spectral sequence
 derived from the filtrations by columns, degenerates to the second page.

To discuss the naturality take a continuous map 
  and let $\mathcal{U}=\{U_i\}_{i\in \mathcal{I}}$ be a good cover of $M$. Find a good cover $\mathcal{V}=\{V_j\}_{j\in \mathcal{J}}$ of $N$  which is a refinement of $f^*\mathcal{U}:=\{f^{-1}(U_i)\}_i$. We have an associated commutative diagram of double complexes with co-augmentations. Passing to the limit the statement follows.
 \end{proof}
   
   \begin{rmk}\label{isocohomology}
	Corresponding to a fixed total order on the index set of the covering we have a smaller bicomplex defined exactly as in \ref{ordering}. This is $\underrightarrow{E}^{p,q}:=\prod_{\sigma \in \underrightarrow{S}_p(\mathcal{U})} C^q_{\textrm{sing}}(U_{\sigma},G),$ $p,q\in \mathbb{N}$ with the same differentials as before. There is a natural projection map of bicomplexes $\pi:E^{p,q} \to \underrightarrow{E}^{p,q}$ commuting with the differentials and with the co-augmentations as well. This is summarized in the following commutative diagram:
	\[\begin{tikzcd}
	{C^*_{\textrm{sing}}(M;G)} \\
	{{E}^{*,*}} & {\underrightarrow{E}^{*,*}} \\
	{\check{C}^*(\mathcal{U};G)} & {\underrightarrow{\check{C}}^*(\mathcal{U};G)}
	\arrow[from=1-1, to=2-1]
	\arrow[from=1-1, to=2-2]
	\arrow["\pi"', from=2-1, to=2-2]
	\arrow[from=3-1, to=2-1]
	\arrow[from=3-1, to=3-2]
	\arrow[from=3-2, to=2-2]
\end{tikzcd}\]

The map $\pi$ induces isomorphism in the
 total cohomology. This boils down the
 fact that for every presheaf the map ${\check{C}}^*(\mathcal{U};\mathcal{F}) \to \underrightarrow{\check{C}}^*(\mathcal{U};\mathcal{F})$ induces isomorphism in cohomology 
   and an application of \cite[Proposition 1.19]{zbMATH03581280}.
\end{rmk}

Now that we have a bicomplex, the ensuing isomorphism $\check{H}^*(\mathcal{U};G) \longrightarrow H^*_{\textrm{sing}}(M;G)$ is given by the familiar Weil's zig-zag procedure. 
We describe it in degree two for it will be useful later:

start with a cocycle $\varphi \in {\underrightarrow{\check{C}}}^2(\mathcal{U};G)$ and find a $\psi \in  \underrightarrow{E}^{1,0}$ 
 such that 
   $\delta \psi=h \varphi$; then lift under $\delta$ the element $d\psi \in \underrightarrow{E}^{1,1}$; we find a $\theta \in \underrightarrow{E}^{0,1}$ with $\delta \theta=d\psi$.
   Finally, as we saw in the proof of Theorem \ref{cohomoth}, the map $\mathsf{r}$ surjects on the subcomplex of the $0$-th column given by 
the kernels of $\delta$ and there induces isomorphism in cohomology.
It follows that there exists $\mu \in C^2_{\textrm{sing}}(M,G)$ with
   $d\mu=0$ and
   $\mathsf{r}\mu=  -d\theta$. Such $\mu$ represents  the image of $[\varphi]$ in $H^2_{\textrm{sing}}(M;G)$. \\
 
Let's look at the case of the sphere that we use in the proof of Theorem \ref{mcmein}.  
Denote with $\Delta^3$ the standard three dimensional simplex with points labelled as in the figure

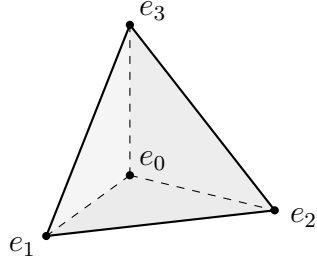
\tdplotsetmaincoords{65}{120}

\begin{center}
\begin{tikzpicture}[tdplot_main_coords, scale=2.2]

\coordinate (O) at (0,0,0);
\coordinate (A) at (0,0,1);
\coordinate (B) at (0,1,0);
\coordinate (C) at (1,0,0);

\fill[gray!25,opacity=.6] (O) -- (B) -- (C) -- cycle;
\fill[gray!15,opacity=.5] (O) -- (A) -- (C) -- cycle;
\fill[gray!35,opacity=.4] (O) -- (A) -- (B) -- cycle;

\draw[dashed] (O) -- (A);
\draw[dashed] (O) -- (B);
\draw[dashed] (O) -- (C);
\draw[thick] (A) -- (B) -- (C) -- cycle;

\foreach \p in {O,A,B,C}
  \fill (\p) circle (0.7pt);

\node at ($(O)+(-0.10,0.1,0.08)$) {$e_0$};
\node at ($(C)+(0.2,-0.05,0)$) {$e_1$};
\node at ($(B)+(0,0.2,0)$) {$e_2$};
\node at ($(A)+(-0.10,0.1,0.08)$) {$e_3$};

\end{tikzpicture}
\end{center}
We fix an orientation preserving homeomorphism $T \to S^2$ and
denote with $F_i$, $i=0,1,2,3$ the opposite face to the vertex $e_i$ (oriented in the usual way) and  with $\mathsf{e}_i$ and $\mathsf{F}_i$ respectively the images of vertices and faces on the sphere.
Consider the open cover $\mathcal{U}=\{U_i\}_{i=0,1,2,3}$
chosen in such a way that $U_i$ is an open neighbourhood of $\mathsf{F}_i$
small enough to make the cover good. 

We want to compute the pairing of the fundamental class $
\mathsf{S}^2:=[\mathsf{F}_0-\mathsf{F}_1+\mathsf{F}_2-\mathsf{F}_3] \in H_2(S^2;\mathbb{Z})$ with 
$\mathscr{C}[\varphi]$ for
a given class in $\check{H}^2(S^2;G)$ represented by a cocycle $(\varphi_{ijk})_{i<j<k}$ in ${\underrightarrow{\check{C}}}^2(\mathcal{U};G)$. 

\begin{prop}\label{fundamentalpairing}
	With the above notation we have, for the pairing between homology and cohomology:
	\begin{equation} \big{\langle} \mathscr{C}[\varphi], \mathsf{S}^2 \rangle = \varphi_{012}(\mathsf{e}_3) - \varphi_{013}(\mathsf{e}_2) + \varphi_{023}(\mathsf{e}_1) - \varphi_{123}(\mathsf{e}_0)\in G.
	\end{equation}
\end{prop}

\begin{proof}

This is a straightforward computation. Keeping the notation
as before in the description of $\mathscr{C}[\varphi]$,
 we have:
\begin{align} \nonumber 
&\big{\langle} \mathscr{C}[\varphi], [\mathsf{F}_0-\mathsf{F}_1+\mathsf{F}_2-\mathsf{F}_3]  \big{\rangle} =  \big{\langle} [\mu], [\mathsf{F}_0-\mathsf{F}_1+\mathsf{F}_2-\mathsf{F}_3]  \big{\rangle}\\ \nonumber &=  \mu_0(\mathsf{F}_0)-\mu_1(\mathsf{F}_1)+\mu_2(\mathsf{F}_2)-\mu_3(\mathsf{F}_3) = -d\theta_0(\mathsf{F}_0)+d\theta_1(\mathsf{F}_1)-d\theta_2(\mathsf{F}_2)+d\theta_3(\mathsf{F}_3).\end{align}

If we denote with $[\mathsf{e}_i,\mathsf{e}_j]$ the image of the edges via the homeomorphism to the sphere, the right hand side of the above expression can be continued as
\begin{align} \nonumber 
&-\theta_0([\mathsf{e}_1,\mathsf{e}_2]+[\mathsf{e}_2,\mathsf{e}_3]- [\mathsf{e}_1,\mathsf{e}_3])  +\theta_1([\mathsf{e}_0,\mathsf{e}_2]+[\mathsf{e}_2,\mathsf{e}_3]- [\mathsf{e}_0,\mathsf{e}_3]) \\  \nonumber
 &-\theta_2([\mathsf{e}_1,\mathsf{e}_3]-[\mathsf{e}_0,\mathsf{e}_3]+ [\mathsf{e}_0,\mathsf{e}_1]) 
 +\theta_3([\mathsf{e}_1,\mathsf{e}_2]-[\mathsf{e}_0,\mathsf{e}_2]+ [\mathsf{e}_0,\mathsf{e}_1])\\ &\nonumber
  = \delta \theta_{(0,3)}([\mathsf{e}_1,\mathsf{e}_2]) + \delta \theta_{01}([\mathsf{e}_2,\mathsf{e}_3]) - \delta \theta_{02}([\mathsf{e}_1,\mathsf{e}_3]) - \delta \theta_{13}([\mathsf{e}_0,\mathsf{e}_2]) \\ &\nonumber + \delta \theta_{12}([\mathsf{e}_0,\mathsf{e}_3]) + \delta \theta_{23}([\mathsf{e}_0,\mathsf{e}_1]). 
\end{align}
Now, since $\delta \theta=d\psi$, the above expression equals
	\begin{align} \nonumber 
&\psi_{03}(\mathsf{e}_2-\mathsf{e}_1)+ \psi_{01}(\mathsf{e}_3-\mathsf{e}_2) - \psi_{02}(\mathsf{e}_3-\mathsf{e}_1) - \psi_{13}(\mathsf{e}_2-\mathsf{e}_0) \\ &\nonumber + \psi_{12}(\mathsf{e}_3-\mathsf{e}_0) + \psi_{23}(\mathsf{e}_1-\mathsf{e}_0) \\ \nonumber &= \delta \psi_{012}(\mathsf{e}_3) - \delta \psi_{013}(\mathsf{e}_2) + \delta \psi_{023}(\mathsf{e}_1) - \delta \psi_{123}(\mathsf{e}_0) \\ & \nonumber = \varphi_{012}(\mathsf{e}_3) - \varphi_{013}(\mathsf{e}_2) + \varphi_{023}(\mathsf{e}_1) - \varphi_{123}(\mathsf{e}_0).
	\end{align}			
\end{proof}

\subsubsection*{\v{C}ech to de Rham}
When the coefficients are real or complex, we also have the \v{C}ech to de Rham bicomplex $\mathcal{E}^{p,q}(M)=\prod \limits_{{\sigma \in S_{p}(\mathcal{U})}}\Omega^q(U_{\sigma})$ equipped with the differential coming from the exterior derivative and two co-augmentations. 

Also we may use smooth singular simplexes and singular cochains, we call it $C^{*,\infty}_{\textrm{sing}}$
to make a smooth version, denoted with $E^{p,q}_{\infty}(M)$ of $E^{p,q}$. Of course the natural map of bicomplexes $E^{p,q}(M) \to E^{p,q}_{\infty}(M)$ induces isomorphism in the total cohomology. 

For any open set $U$, the de Rham map $\mathcal{I}: \Omega^q(U) \to C^{q,\infty}_{\operatorname{sing}}(U;\mathbb{R})$ defined by
$$\mathcal{I}(\omega) (\mu)= \int_{\Delta^q}\mu^*\omega \quad \textrm{for every smooth singular simplex }\mu: \Delta^p \to U,$$ gives a map of bicomplexes $\mathcal{I}:\mathcal{E}^{p,q}(M) \to E^{p,q}_{\infty}(M)$ commuting with the co-augmentations and inducing isomorphisms in the total cohomology. In particular the de Rham isomorphism enters in a commutative diagram 
\[\begin{tikzcd}
	{H^*_{\textrm{dR}}(M;\mathbb{R})} & {H^{*,\infty}_{\textrm{sing}}(M,\mathbb{R})} \\
	{\check{H}^*(M;\mathbb{R})}
	\arrow["{\mathcal{I}}", from=1-1, to=1-2]
	\arrow["\cong"', from=1-1, to=2-1]
	\arrow["\cong", from=1-2, to=2-1]
\end{tikzcd}\]
of natural isomorphisms.
\section{Darboux derivatives}\label{darboux}
We collect some basic facts about Maurer--Cartan forms and Darboux derivatives (cf. \cite[Chapter 5]{MACKENZIE2000445}). \smallskip 

Let $G$ be a Lie group; 
the right invariant Maurer--Cartan form $\theta^{R} \in \Omega^1(G,\mathfrak{g})$ is the unique right invariant $\mathfrak{g}$-valued form such that $\theta^R_{e}=\operatorname{Id}$.
Explicitly $\theta_g(u)=d_gR_{g^{-1}}(u).$ 
For $u \in T_gG$. It can be characterized as the unique form such that $\theta^R(X^R)=X$ for every $X\in \mathfrak{g}$.
Denoting with $\iota:G \to G$ the inversion $\iota(g)=g^{-1}$ we have the following properties:
\begin{itemize}
\item $(\iota^*\theta^R)_g=-\operatorname{Ad}_{g^{-1}}\circ\, \theta_g^R$.
\item The Maurer--Cartan equation: 
\begin{equation*}
d\theta^R(X,Y)+ [\theta^R(X),\theta^R(Y)]=0
\end{equation*}
 for $X,Y\in T_gG$.
\end{itemize}

We say that a Maurer--Cartan form on a manifold is a one form $\omega\in \Omega^1(M;\g)$ such that 
$$d\omega + [\omega,\omega]=0.$$

For a smooth map $f:M \to G$, the right Darboux derivative of $f$ is $$\Delta(f):=f^*\theta^R \in \Omega^1(M;\g).$$ Then $d\Delta(f)+[\Delta(f),\Delta(f)]=0$ i.e. $\Delta(f)$ is a Maurer--Cartan form and satisfies:
\begin{itemize}
\item product formula $\Delta(f_1f_2)=\Delta(f_1)+\operatorname{Ad}_{f_1}(\Delta(f_2))$.
\item $\Delta(f^{-1})=-\operatorname{Ad}_{f^{-1}}(\Delta(f))$.
\item If $\Delta(f_1)=\Delta(f_2)$ and $M$ is connected then $f_1^{-1}f_2$ is constant; there is a $g\in G$ such that $f_2=R_g \circ f_1$. If $m\in M$ is a reference point: $g=f_1^{-1}(m)f_2(m)$.
\end{itemize}

\begin{prop}
Let $f:M \to G$ be a smooth map and consider the map $\operatorname{Ad}_f:M \to \operatorname{End}(\mathfrak{g})$. We have 
$$d\operatorname{Ad}_f=[\Delta(f), \operatorname{Ad}_f]=\operatorname{ad}^{\ell}_{\Delta(f)}\circ  \operatorname{Ad}_f=-\operatorname{ad}^{\mathsf{r}}_{\Delta(f)}\circ \operatorname{Ad}_f.$$
Explicitly $$(d_m\operatorname{Ad}_{f(m)}(v)) (X)=[\Delta(f)(v), \operatorname{Ad}_{f(m)}(X)]$$
for $v\in T_mM$ and $X\in \mathfrak{g}$.
\end{prop}

\begin{thm}{\em{(The fundamental Theorem of calculus)}}\label{mcintegral}
Let $G$ be a Lie group with Lie algebra $\mathfrak{g}$ and let 
$M$ be a simply connected manifold. Let
$\omega \in \Omega^1(M;\mathfrak{g})$ be a Maurer--Cartan form.
There is a smooth map $f:M \to G$ with $\Delta(f)=\omega$.
Such map is unique up to right translation with an element of $G$, i.e. if $\Delta(f_1)=\Delta(f_2)$ then there is a unique $g\in G$ such that $f_2=R_g \circ f_1$. 

In particular, with a fixed reference point $m$, we have a unique integral satisfying $f(m)=1$.
\end{thm}
When $M$ is connected but non necessarily simply connected, let $\pi:\widetilde{M}\to M$ be the universal cover of $M$ with base point $m$
and right action of $\pi_1(M,m)$. Let $\omega$ be a Maurer--Cartan form on $M$; pulling back $\omega$ to $\widetilde{M}$ 
and fixing a point $\widetilde{m}$ over $m$, we can find $\widetilde{f}:\widetilde{M}\to G$ with
$$\Delta(\widetilde{f})=\omega, \quad \widetilde{f}({\widetilde{m}})=1$$
and a monodromy map $\rho:\pi_1(M,m) \to G$ such that
$$\widetilde{f}\circ R_{\gamma}=R_{\rho(\gamma)}\circ \widetilde{f},$$
for every $\gamma \in \pi_1(M,m)$.
The monodromy is a morphism of groups $\rho(\gamma_1\gamma_2)=\rho(\gamma_1)\rho( \gamma_2)$, $\gamma_1,\gamma_2 \in \pi_1(M,m)$.

Indeed first we solve the integration problem on $\widetilde{M}$ for the form $\pi^*\omega$ to find $\widetilde{f}$ (with the base point condition); then we compare any $\widetilde{f}\circ R_{\gamma}$ with $\widetilde{f}$. For every $\gamma \in \pi_1(M,m)$ there is a $\rho(\gamma) \in G$ with $\widetilde{f} \circ R_{\gamma}=R_{\rho(\gamma)}\circ \widetilde{f}$. This $\rho(\gamma)$ is uniquely determined by $\gamma$; indeed $\rho(\gamma)=\widetilde{f}(\widetilde{m} \cdot \gamma)$. 
This means that $\widetilde{f}$ is quasi periodic:
$$\widetilde{f}(x \cdot \gamma)=\widetilde{f}(x)\rho(\gamma) $$
for every $x\in \widetilde{M}$ and $\gamma \in \pi_1(M,m)$.
That $\rho$ is a morphism is automatic.

\bibliography{MacK}
\bibliographystyle{plain}

\end{document}